\documentclass[10pt, letterpaper]{article}
\usepackage[utf8]{inputenc}
\usepackage{amsmath}
\usepackage{amsthm}
\usepackage{mathtools}
\usepackage{mathrsfs} %\mathscr{}
\usepackage{amsfonts}
\usepackage{amssymb}
\usepackage{amsthm}
\usepackage{indentfirst}

\usepackage{graphicx}
\usepackage{subcaption}
\usepackage{colortbl}
\usepackage{enumitem}

\usepackage{authblk}

\usepackage{url}

\usepackage{tikz}

\usepackage{xcolor}
\usepackage{color, soul}
\usepackage{stmaryrd}
\usepackage{bookmark}
\usepackage{booktabs}
\usepackage{stackengine}
\usepackage{stmaryrd}

\newcommand{\doublecurly}[1]{%
  \mathopen{\scalebox{1.0}{$\lbrace\!\mkern-1mu\lbrace$}}#1\mathclose{\scalebox{1.0}{$\rbrace\!\mkern-1mu\rbrace$}}%
}

\newcommand{\dc}[1]{\doublecurly{#1}}

\graphicspath{{./images/}}

\def\softd{{\leavevmode\setbox1=\hbox{d}%
		\hbox to 1.05\wd1{d\kern-0.4ex{\char039}\hss}}}%cstocs

\newcommand{\brho}{\bar \rho}
\newcommand{\bc}{\bar c}
\newcommand{\grid}{\mathcal{T}}

\newcommand{\jump}[1]{\ensuremath{\left[\![ #1\right]\!]}}

\newcommand {\startenv} {\vskip 0.05em
\begin{tabular}{||l}\parbox[t]{0.95\linewidth}}
\newcommand {\stopenv} {\end{tabular}\vskip 0.05em}

\newtheorem{theorem}{Theorem}[section]
\newtheorem{definition}[theorem]{Definition}

\newtheorem{lemma}[theorem]{Lemma}
\newtheorem{corollary}[theorem]{Corollary}

\newtheorem{assumption}[theorem]{Assumption}
\newtheorem{remark}[theorem]{Remark}
\newenvironment{prooflemma}[1]
  {\renewcommand{\proofname}{Proof of Lemma #1}\proof}
  {\endproof}

\newcommand{\keywords}[1]{\textbf{Keywords:} #1}

\author[1]{Jan Giesselmann \thanks{jan.giesselmann@tu-darmstadt.de}}
\author[2]{Kiwoong Kwon\thanks{Corresponding author; kwkwon@knu.ac.kr} }
\date{}

\affil[1]{\centering Department of Mathematics\\ Technical University of Darmstadt\\ Dolivostr. 15, 64293 
Darmstadt, Germany}
\affil[2]{\centering Department of Mathematics\\ Kyungpook National University\\ Daegu, Republic of Korea}

\title{A posteriori error control for a discontinuous Galerkin approximation of a Keller-Segel model}

\begin{document}
% \listoftodos
% TODO:
% \begin{itemize}
%     \item Optimize the number of references
% \end{itemize}

\maketitle

\begin{abstract}
We provide a posteriori error estimates for a discontinuous Galerkin scheme for the parabolic-elliptic Keller-Segel system in 2 or 3 space dimensions. The estimates are conditional in the sense that an a posteriori computable quantity needs to be small enough - which can be ensured by mesh refinement - and optimal in the sense that the error estimator decays with the same order as the error under mesh refinement. A specific feature of our error estimator is that it can be used to prove the existence of a weak solution up to a certain time based on numerical results.
\end{abstract}

\keywords{Keller-Segel; chemotaxis; nonlinear diffusion; discontinuous 
Galerkin scheme; a posteriori error analysis}

{\bf AMS subject classification} (2020):
Primary 65M50,
Secondary 65M08, 35K40

\section{Introduction}
The Keller-Segel system is a well-known mathematical model in the field of chemotaxis.
%widely studied in the literature \cite{Childress_1981, Nagai1998, Biler1998, Senba2002, Suzuki2005, Blanchet2006, 2014Campos, 2016Fernandez}. 
Introduced by E. Keller and L. Segel \cite{Keller_1970}, it describes the collective motion of cells in response to concentration gradients. It consists of two coupled partial differential equations:
\begin{equation} \label{intro:pp_eqn}
    \begin{matrix}
        \begin{aligned}
            \partial_t\rho +\nabla\cdot(\chi\rho\nabla c - \nabla \rho) &= 0 \\
            \epsilon\partial_t c + c - \Delta c - \alpha\rho &= 0
        \end{aligned}
    \end{matrix}
\end{equation}
and suitable initial and boundary conditions.
Here $\rho=\rho(t,x)$ and $c=c(t,x)$ denote the density of the cell population and the concentration of chemically attracting substances, respectively, for any  time $t>0$ at location $x \in \Omega$.
We will consider the problem on a bounded open set $\Omega$ in $\mathbb{R}^d$ with a piecewise smooth boundary $\partial \Omega$. The parameters $\chi$, $\alpha$, and $\epsilon$ are constant and  satisfy $\chi, \alpha >0$ and $\epsilon\geq 0$.  The modelling background of this system has been extensively discussed in the literature \cite{Hillen2008, Murray2002}, and there are many extensions, e.g., models accounting for multiple stimuli and  multiple species  \cite{2002Wolansky, 2010Horstmann, 2020Karmakar}.

In this work, we focus on the parabolic-elliptic form of the Keller-Segel system described by the partial differential equations \eqref{intro:pp_eqn} when $\epsilon = 0$ (and setting the other parameters to $1$ for simplicity):
\begin{eqnarray} 
 \partial_t\rho +\nabla\cdot(\rho\nabla c - \nabla \rho) &= 0 \quad\text{in} \ (0,T)\times\Omega  \label{eq:KS1}\\
            c - \Delta c &= \rho\quad 
    \text{in} \ (0,T)\times\Omega \label{eq:KS2}\\
 \nabla \rho \cdot \mathrm{n} &=0\quad \text{on} \ (0,T) \times \partial \Omega\label{eq:bc1}\\
\nabla c \cdot \mathrm{n} &=0 \quad \text{on} \ (0,T) \times\partial \Omega \label{eq:bc2}\\
\rho(0, \cdot) &= \rho_0  \quad  \text{in} \ \Omega \label{eq:ic}
\end{eqnarray}
where $\mathrm{n}$ denotes the outward unit normal vector.

A striking feature of the Keller-Segel system which motivates detailed investigations is the finite time blow-up of solutions \cite{Nagai1995}. Such singularities are expected in settings of dimension $d \geq 2$  when the initial cell population surpasses a certain threshold. For $d=2$, this threshold is on the total mass \cite{Gajewski_1998, Blanchet2006} and for $d \geq 3$, it depends on the maximal local concentration. For an overview on the Keller-Segel system, we refer to the review papers \cite{Horstmann2003, Horstmann2004}.

For special initial data explicit formulas for blow up solutions exist but, in general, approximating  solutions of this system with blow-up accurately is challenging. Several numerical methods have been developed over the years to tackle this issue. These methods include a fractional step method \cite{Tyson2000}, finite volume schemes \cite{Filbet2006, 2008Chertock}, and finite-element methods \cite{Saito_2007, 2013Strehl}. Additionally, \cite{Epshteyn2008/09, Li2017} employed discontinuous Galerkin methods, and \cite{2005Budd} utilized a moving mesh method.

More generally, solutions to chemotaxis systems frequently display strong growth of cell density close to some point or curve. Indeed, the system may not only exhibit blow-up but it may also develop other `spiky' structures.
%, which is referred to as the \emph{chemotactic collapse} \cite{Senba2002}. This phenomenon is characterized by a concentration of $\rho$ akin to a Dirac measure.
Resolving such highly localized structures on uniform meshes is, arguably, inefficient and there has been an intense interest in the development of (mesh) adaptive  numerical schemes \cite{Chertock2019, Sulman2019}
% with a strong focus on mesh adaptation.
with the goal of increasing  accuracy  and efficiency.
Several heuristic strategies for mesh adaptation can be found in the literature. For example, \cite{Sulman2019} presents an adaptive moving mesh finite element method that uses a coordinate transformation to concentrate grid nodes in regions of large solution variations. Similarly, \cite{Chertock2019} proposes a semi-discrete adaptive moving mesh finite-volume upwind method, which enhances resolution in blow-up regions by increasing the density of mesh nodes.
To the best of our knowledge, mesh adaptation based on a posteriori error estimates has not been investigated. The goal of this work is to provide a basis for such investigations. At the same time, our results provide rigorous error control of simulations.

Mesh adaptation  based on a posteriori error estimates has been extremely successful for simpler types of PDEs, even leading to provably optimal meshes in certain cases. For instance, convergence of an adaptive space-time finite element method that relies on an a posteriori error estimator for solving linear parabolic partial differential equations  was proven in \cite{Kreuzer_2012}. A posteriori error estimates come in several varieties: If there is a goal functional of specific interest, dual weighted residuals may be used (see, e.g., \cite{Becker2001}), which usually leads to very efficient meshes.
We follow a different approach that aims at controlling the error in a suitable norm and is nicely explained in \cite{Makridakis2007}. The main idea is to insert a (sufficiently regular) reconstruction of the numerical solution into the PDE so that a suitable stability theory can be used to bound the difference between the exact solution to the PDE and this reconstruction. For nonlinear equations it may happen that the stability theory allows only \emph{conditional} a posteriori error estimates, e.g., Allen-Cahn equation \cite{Bartels}, and this is indeed the case here.
 
Different stability frameworks, e.g., ones based on (relative) entropy or negative order norms, could be considered for the Keller-Segel system. We  present an a posteriori error estimate using a rather standard $L^2$-based stability framework since this is the framework that, for us, leads to the strongest results. It might seem that using an $L^2$-based stability framework requires unrealistic and unverifiable  assumptions on the regularity of exact solutions but it turns out that, for sufficiently regular initial data, e.g., $\rho_0$ in $L^2(\Omega)$, weak solutions have this regularity until blow-up time \cite{Biler1994} and we provide an a posteriori verifiable condition that guarantees that the exact solution is sufficiently regular.

%if the a posteriori condition that ensures validity of our estimator holds this proves that the exact solution does not exhibit blow-up.

The numerical scheme studied in this manuscript is a  straightforward discontinuous Galerkin (dG) scheme. Much more sophisticated schemes have been developed in the literature, e.g., \cite{Guo2019,Qiu_2021}, offering features such as entropy dissipation and positivity preservation. It is beyond the scope of the current work to extend our error estimator to those schemes but it should be noted that several results from this paper, in particular Theorem \ref{thm:stabilityest}, can be directly used for the  analysis of those schemes,
e.g., the companion paper \cite{GiesselmannKolbe} provides suitable reconstructions and an $H^{-1}$-norm estimate for the residual for a positivity preserving finite volume scheme.

Our analysis is performed in a fully discretised setting.
It is worth noting that the error estimator presented in our work is of optimal order, i.e. the error estimator has the same order of convergence as the norm of the error that it controls.

We investigate a discontinuous Galerkin (dG) scheme discretising Laplace operators by the symmetric interior penalty (SIP) bilinear form and the chemotaxis term by the weighted SIP (wSIP) bilinear form treating density as a diffusion coefficient.
For time discretisation, we use an IMEX (implicit-explicit) scheme.
Our analysis applies to arbitrary polynomial degrees $ k \geq 1$.

The outline of the paper is as follows: In Section \ref{section:background}, we provide some background on weak solutions for the Keller-Segel system with a focus on blow-up phenomena. In Section \ref{section:a_stability_framework}, we establish a stability framework for the Keller-Segel system, which serves as a crucial component in obtaining our a posteriori error estimator. Notably, the stability estimate obtained here can, in certain situations, verify the existence of a sufficiently regular weak solution a posteriori. In Section \ref{sec:dgscheme}, we introduce our discontinuous Galerkin scheme. In Section \ref{sec:reconstruction}, we define reconstructions of the numerical solutions, utilizing the so-called elliptic reconstruction, and also define the residual of the PDEs. In Section \ref{section:aposteriori}, we derive a computable upper bound for the $H^{-1}$-norm of the residual. In Section \ref{sec:aposteriori}, we present the a posteriori error estimator based on the analysis conducted in the previous sections. Finally, in Section \ref{sec:numerical}, we show the results of numerical experiments to verify the optimal scaling behavior of the error estimator.

\section{Background on weak solutions and blow-up} \label{section:background}

We study the parabolic-elliptic Keller-Segel system with the homogeneous Neumann boundary conditions \eqref{eq:KS1}--\eqref{eq:ic}.
%
%of parabolic-elliptic partial differential equations
%
% \begin{equation} \label{eq:KS}
%     \begin{matrix}
%         \begin{aligned}
%             \partial_t\rho +\nabla\cdot(\rho\nabla c - \nabla \rho) &= 0 \\
%             c - \Delta c &= \rho
%         \end{aligned}
%         & \text{in} \ (0,T)\times\Omega,
%     \end{matrix}
% \end{equation}
% where $\Omega$ is a bounded open set in $\mathbb{R}^d$ $(d=2,3)$ with piecewise smooth boundary $\partial \Omega$. The homogeneous Neumann boundary conditions are imposed
% \begin{equation} \label{eq:Neumann}
%     \begin{matrix}
%         \begin{aligned}
%             \frac{\partial\rho}{\partial \mathrm{n}} &:= \nabla \rho \cdot \mathrm{n} = 0 \\
%             \frac{\partial c}{\partial \mathrm{n}} &:= \nabla c \cdot \mathrm{n} = 0
%         \end{aligned} & \ \text{on} \ (0,T) \times\partial \Omega
%     \end{matrix}
% \end{equation}
% where $\mathrm{n}$ is the outward unit normal, and the initial condition 
% \begin{equation} \label{eq:Init}
%     \rho(0, \cdot) = \rho_0
% \end{equation}
% is given. 
%
%
The existence of a local-in-time weak solution for 
this problem
%defined by equations \eqref{eq:KS}-\eqref{eq:Init} on $\left( 0, T \right)\times \Omega$
 is well established \cite[Section 3]{Biler1994}.
Let us recall the definition of a weak solution as given in \cite{Biler1994}:

\begin{definition} \cite{Biler1994} \label{def:weaksol}
    A function $\rho \in L^{\infty}\left(0,T ; L^2(\Omega)\right) \cap L^2\left(0,T ; H^1(\Omega)\right)$ is called a \emph{weak solution} of the problem \eqref{eq:KS1}--\eqref{eq:ic} on $\left( 0, T \right)\times \Omega$, if it satisfies
    \begin{multline*}
        \int_{\Omega} \rho(t, x) \varphi(t, x) d x - \int_0^t \int_{\Omega} \rho \partial_t \varphi - (\nabla\rho - \rho \nabla c) \cdot \nabla \varphi dx ds\\
        =\int_{\Omega} \rho_0(x) \varphi(0, x) d x
   \end{multline*}
    for every $\varphi \in H^1((0, T)\times\Omega )$ and for a.e. $t \in(0, T)$
    and, if,
    in addition, for a.e. $t \in(0, T)$
        \begin{equation*}
        c(t, \cdot) \in H^1(\Omega) \text { with } c(t, x) = (G * \rho(t,\cdot))\left( x \right) \quad \text{for a.e.}\ x \in \Omega
        \end{equation*}
        holds, 
    where $G$ is an appropriate Green's function for $\Omega$.
\end{definition}

\begin{theorem}[Existence of weak solutions] \cite[Theorem 2]{Biler1994} \label{thm:weaksol}
    Assume that $\Omega$ is a bounded domain in $\mathbb{R}^d$ with piecewise smooth boundary.
    \begin{enumerate} [label=(\roman*)]
        \item If $d\in\{2,3\}$, and $0 \leq \rho_0 \in L^2(\Omega)$, then there exists $T=$ $T\left( \left\| \rho_0 \right\|_{L^2(\Omega)} \right)$ such that the problem \eqref{eq:KS1}--\eqref{eq:ic} has a unique weak solution $\rho \in$ $L^{\infty}\left(0, T; L^2(\Omega)\right) \cap L^2\left(0, T ; H^1(\Omega)\right)$. Moreover, $\partial_t\rho \in L^2\left(0, T ; H^{-1}(\Omega)\right)$, $\rho(t, x) \geq 0$ for a.e. $x \in \Omega$ and a.e. $t \geq 0$, and $\int_{\Omega} \rho(t, x) d x=\int_{\Omega} \rho_0(x) d x$.
        \item If $d \geq 2$ and $0 \leq \rho_0 \in L^p(\Omega)$ with $ p>d / 2$, then there exists $T=$ $T\left(p, \left\| \rho_0 \right\|_{L^p(\Omega)}\right)>0$ and a weak solution $\rho$ such that $\rho \in L^{\infty}\left(0, T ; L^p(\Omega)\right)$ and $\rho^{p / 2} \in L^2\left(0, T ; H^1(\Omega)\right)$.
        These solutions are unique when $p>d$, and regular when $p>d / 2$ in the sense that $\rho \in L_{\mathrm{loc}}^{\infty}\left(0, T ; L^{\infty}(\Omega)\right)$.
    \end{enumerate}
\end{theorem}

\begin{remark} \label{rmk:energy_eq}
    With (i) of Theorem \ref{thm:weaksol}, a routine calculation, e.g., \cite{Biler1998} and references therein, shows that such a weak solution satisfies $\rho \in C\left([0, T] ; L^2(\Omega)\right)$, and the energy balance 
    \begin{equation*}
        \frac{1}{2} \int_{\Omega} \rho^2(t, x) d x+\int_0^t \int_{\Omega}(\nabla \rho - \rho \nabla c) \cdot \nabla \rho dx ds = \frac{1}{2} \int_{\Omega} \rho_0^2(x) d x.
    \end{equation*}
    Moreover, by the representation $c = G * \rho$, we know that $c \in C([0,T]; H^1(\Omega))$.

    For notational convenience, we abuse the differential form throughout this paper:
    \begin{equation} \nonumber
        \frac{1}{2}\frac{d}{dt} \left\| \rho(t, \cdot) \right\|_{L^2(\Omega)}^2 + \left\| \nabla \rho(t, \cdot) \right\|_{L^2(\Omega)}^2 = \int_\Omega \rho(t, \cdot) \nabla c(t, \cdot) \cdot \nabla \rho(t, \cdot)dx .
    \end{equation}
    This is obtained through a formal process that involves computing the $L^2$ inner product of the equation with $\rho$, followed by an integration by parts.

\end{remark}
We make several observations regarding the features of blow-up of solutions, which will play a crucial role in our a posteriori analysis in Section \ref{section:a_stability_framework}.
\begin{remark} [Blow-up of $L^p$-norms] \label{rmk:uniformbound}
    We remark that (ii) of Theorem \eqref{thm:weaksol} means that if a solution blows up at some time $T>0$ (i.e. $\lim_{t \nearrow T} \left\| \rho(t, \cdot) \right\|_{L^\infty(\Omega)} = \infty$), then the norms $\left\| \rho(t, \cdot) \right\|_{L^p(\Omega)}$ blow up for all $p \in(d / 2, \infty]$ at the same moment \cite[p.333]{Biler1994}.
\end{remark}

\begin{lemma} [Blow-up criterion for weak solutions]
    \label{rmk:blow-up}
    Let $\rho_0 \in L^2(\Omega)$ and let $T_{max} \in (0, +\infty]$ be the maximal existence time of the weak solution $\rho$ for \eqref{eq:KS1}--\eqref{eq:ic}, i.e., the supremum of all $T>0$ such that the weak solution $\rho$ exists on $[0, T]$.
    If $T_{max}<\infty$, we have 
    \begin{equation} \label{eq:blow-up}
        \lim_{t \nearrow T_{max}}\left\| \rho(t, \cdot) \right\|_{L^\infty(\Omega)} = \infty.
    \end{equation}
\end{lemma}

The existing literature, e.g., \cite[Theorem 3.2]{Suzuki2005}, provides a proof of the blow-up criterion for smooth solutions to the Keller-Segel system. However, despite the possibility of establishing this property for weak solutions, to the best of our knowledge, it is not available in literature. Since this property is crucial for the proof of Corollary \ref{thm:guarantee_existence}, we provide a rigorous proof of Lemma \ref{rmk:blow-up}.

\begin{prooflemma}{2.5}
Assume that equation \eqref{eq:blow-up} does not hold. We can then extract a sequence $\left\{ t_k  \right\}_{k=1}^{\infty}$, such that $\left\| \rho(t_k, \cdot) \right\|_{L^\infty(\Omega)} < C$ for all $k$, and where $t_k$ approaches $T_{max}$ from below as $k$ tends to infinity. Since $\Omega$ is bounded, we have $\left\| \rho(t_k, \cdot) \right\|_{L^2(\Omega)} \leq C$ for all $k$. Now if we use each $t_k$ as an initial time point, then $(i)$ of Theorem \ref{thm:weaksol} enables us to choose a number $\delta_k = \delta_k\left( \left\| \rho(t_k, \cdot) \right\|_{L^2(\Omega)} \right) > 0$ such that the weak solution $\rho$ exists on the time interval $[t_k, t_k + \delta_k)$.
Moreover, since the sequence $\left\| \rho(t_k, \cdot) \right\|_{L^2(\Omega)}$ is uniformly bounded, the sequence $\left\{ \delta_k \right\}$ does not shrink to zero (cf. \cite[the proof of Theorem 1]{Biler1992}). Thus we have $\delta : = \inf_{k}\delta_k > 0$. Now, we can choose $k_0 \in \mathbb{N}$ such that $T_{max} - \frac{1}{2}\delta < t_{k_0}$. By employing the continuation argument, we find that the weak solution $\rho$ exists until $t_{k_0} + \delta > T_{max}$, a contradiction to the definition of $T_{max}$. Thus, \eqref{eq:blow-up} holds true.
\end{prooflemma}

The stability framework that will be presented in the next section requires that $c$ can  be controlled by sufficiently  weak norms of  $\rho$, i.e. elliptic regularity for \eqref{eq:KS2}, \eqref{eq:bc2}. Since \eqref{eq:KS2} has constant coefficients, this depends on properties of $\Omega$. A sufficient condition is that $\Omega$ is convex.
%The domain $\Omega$ needs to be regular enough for the Poisson problems to have elliptic regularity. From now on we make the following assumption:
\begin{assumption} \label{assumption:elliptic}
    Throughout this paper, we assume $d \in \{2, 3\}$ and that $\Omega$ is 
    such that \eqref{eq:KS2}, \eqref{eq:bc2} enjoys elliptic regularity, i.e. there exists a positive constant $C_{ell}>0$, depending only on $\Omega$, such that $\|  c(t, \cdot)\|_{H^2} \leq C_{ell} \| \rho \left( t, \cdot \right)\|_{L^2}$ for all $t$.
\end{assumption}

\section{Stability framework} \label{section:a_stability_framework}
Let  $\left( \rho, c \right)$ be a weak solution to the problem \eqref{eq:KS1}--\eqref{eq:ic} and let $\left( \bar{\rho}, \bar{c} \right)$ with $\bar{\rho} \in C\left( [0,T]; H^1(\Omega) \right)$, $\partial_t \bar{\rho} \in C\left( (0,T]; H^{-1}(\Omega) \right)$, $\bar{c}\in C\left( [0,T]; H^1(\Omega) \right)$ be a \emph{strong} solution to the following perturbed problem:
\begin{eqnarray}
    \partial_t \brho + \nabla \cdot (\brho \nabla \bc - \nabla \brho) &=R_\rho   \quad \text{ in } (0,T) \times \Omega\label{eq:KS-p1}\\
    -  \Delta \bc + \bc &= \brho     \quad \text{ in } (0,T) \times \Omega\label{eq:KS-p2}\\
   \nabla \rho \cdot \mathrm{n} &=0 \quad \text{ on } (0,T) \times \partial \Omega\\
   \nabla c \cdot \mathrm{n} &=0 \quad \text{ on } (0,T) \times \partial \Omega
   \label{eq:bc-p2}
\end{eqnarray}
with some given function $R_\rho \in L^2(0,T; H^{-1}( \Omega))$.
In this section, we provide an estimate for the difference $(\rho - \brho, c - \bc)$ in terms of the (possible) difference of initial data and  of $R_\rho$.
The situation we have in mind is that $(\brho, \bc)$ is obtained as a reconstruction of a numerical solution, see Section \ref{sec:reconstruction}. Nevertheless, we should stress that our stability framework does not depend on how $(\brho, \bc)$ is obtained.
 
Taking into account Remark \ref{rmk:energy_eq}, we will subsequently provide a formal argument for the derivation of the stability estimate, but it can be made rigorous by interpreting it  in the context of appropriate integral formulation. For the sake of brevity, we choose not to explicitly include $\Omega$ in the norm symbols. Furthermore, the time dependency of functions will also not be explicitly denoted. 

Subtracting \eqref{eq:KS-p1} from \eqref{eq:KS1} and testing with $\rho - \bar{\rho}$ gives
%\begin{equation}
 %\int_{\Omega} (\rho - \brho) \partial_t (\rho - \brho)  = \int_{\Omega} (\rho - \brho) \Delta (\rho -\rho) -  (\rho - \brho) \nabla \cdot \left( \rho \nabla c - \brho \nabla \bc\right) - R_\rho(\rho - \brho)
%\end{equation}
%which after integration by parts implies
%\begin{multline}
 %\frac{d}{dt} \left[ \int_{\Omega} \frac{1}{2} (\rho - \brho)^2 dx \right]  +%\int_{\Omega}  |\nabla (\rho - \brho)|^2 dx \\
 %= \int_\Omega \nabla (\rho - \brho) \rho \nabla (c - \bc) + \nabla (\rho - \brho) (\rho - \brho) \nabla \bc - R_\rho(\rho - \brho) dx
 %\end{multline}
%which  can be rewritten as
 \begin{multline*}
 \frac{d}{dt} \left[ \int_{\Omega} \frac{1}{2} (\rho - \brho)^2  dx \right] +\int_{\Omega}   |\nabla (\rho - \brho)|^2 dx\\
 = \int_\Omega \nabla (\rho - \brho) \brho \nabla (c - \bc) + \nabla (\rho - \brho) (\rho - \brho) \nabla \bc - R_\rho(\rho - \brho)
 + \nabla (\rho - \brho) (\rho -\brho) \nabla (c - \bc) dx,
 \end{multline*}
where we used the homogeneous boundary conditions. Using Cauchy-Schwarz's inequality we obtain
\begin{multline*}
        \frac{d}{dt} \left[\frac 12 \|  \rho - \brho\|_{L^2}^2  \right]
        +   |\rho - \brho|_{H^1}^2  \\
        \leq  |\rho - \brho|_{H^1} \|\brho\|_{L^3}  |c - \bc|_{W^{1,6}} + |\rho - \brho|_{H^1} \|\rho - \brho\|_{L^2} \|\nabla \bc\|_{L^\infty}\\
         \quad + \|R_\rho\|_{H^{-1}} \|\rho - \brho\|_{H^1}
        + |\rho - \brho|_{H^1} \|\rho -\brho\|_{L^3} |c - \bc|_{W^{1,6}}.
\end{multline*}
Since the number of space dimensions satisfies $d \leq 3$, we have $\left| c - \bar{c} \right|_{W^{1,6}} \leq C_S \left\| \nabla c - \nabla \bar{c} \right\|_{H^1}$ and 
$\left\| \rho - \bar{\rho} \right\|_{L^3} \leq C_S' \left\| \rho - \bar{\rho} \right\|_{H^1}$,
where $C_S$ and $C_S'$ are the constants of the embedding $H^1 \rightarrow L^6$ and $H^1 \rightarrow L^3$, respectively. Moreover, by elliptic regularity, we have $\left\| c - \bar{c} \right\|_{H^2} \leq C_{ell} \left\| \rho - \bar{\rho} \right\|_{L^2}$. % where $C_{ell}$ is the constant of elliptic regularity.
\begin{multline*}
    \frac{d}{dt} \left[\frac 12 \|  \rho - \brho\|_{L^2}^2  \right]  +  |\rho - \brho|_{H^1}^2 \nonumber \\
    \leq C_SC_{ell} |\rho - \brho|_{H^1} \|\brho\|_{L^3} \left\| \rho - \brho \right\|_{L^2} + |\rho - \brho|_{H^1} \|\rho - \brho\|_{L^2} \|\nabla \bc\|_{L^\infty}\nonumber\\
      + \|R_\rho\|_{H^{-1}} \|\rho - \brho\|_{H^1}
    +C_S' C_S C_{ell} |\rho - \brho|_{H^1} \|\rho -\brho\|_{H^1} \left\| \rho - \bar{\rho} \right\|_{L^2}.
\end{multline*}
Using Young's inequality and gathering terms on the right hand side, we obtain
\begin{multline}\label{3}
    \frac{d}{dt} \left[ \left\| \rho - \bar{\rho} \right\|_{L^2}^2 \right] + \left| \rho - \bar{\rho} \right|_{H^1}^2 \leq \left( 3 C_S^2 C_{ell}^2 \left\| \bar{\rho} \right\|_{L^3}^2 + 3 \left\| \nabla \bar{c} \right\|_{L^\infty}^2 + \frac{1}{3} \right) \left\| \rho - \bar{\rho} \right\|_{L^2}^2 \\ + 3 \left\| R_\rho \right\|_{H^{-1}}^2
    + 2C_S'C_S C_{ell} |\rho - \brho|_{H^1} \| \rho - \brho\|_{H^1} \|\rho -\brho\|_{L^2}.
\end{multline}
 Let us set
 \begin{equation} \label{eq:y123a}
 \begin{split}
  y_1 (t) &:= \|  \rho (t,\cdot) - \brho (t,\cdot)\|_{L^2}^2,   \\
  y_2(t) &:= |  \rho (t,\cdot) - \brho (t,\cdot)|_{H^1}^2, \\
  y_3 (t)&:= 2C_S'C_SC_{ell} |\rho(t,\cdot) - \brho(t,\cdot)|_{H^1} \| \rho(t,\cdot) - \brho(t,\cdot)\|_{H^1} \|\rho(t,\cdot) -\brho(t,\cdot)\|_{L^2},\\
  a(t) &:= 3 C_S^2C_{ell}^2 \|\brho(t,\cdot)\|_{L^3}^2  + 3 \|\nabla \bc(t,\cdot)\|_{L^\infty}^2 + \frac{1}{3}.
 \end{split}
\end{equation}
Then, we can integrate \eqref{3} in time from $0$ to $T'$ to obtain
\begin{equation}\label{near:gron}
 y_1(T') +  \int_0^{T'} y_2(t) dt 
 \leq
 y_1(0)  + \int_0^{T'}\| R_\rho\|_{H^{-1}}^2 dt + \int_0^{T'} a(t) y_1(t) dt + \int_0^{T'} y_3(t)  dt
\end{equation}
Since $ \left\| \rho - \bar{\rho} \right\|_{H^1}^2 = \left\| \rho - \bar{\rho} \right\|_{L^2}^2 + \left| \rho - \bar{\rho} \right|_{H^1}^2 $,
we have
\begin{multline}  \nonumber
    y_3(t) \leq
    2C_S'C_SC_{ell} %\left( |\rho - \brho|_{H^1} \|\rho -\brho\|_{L^2}^2 +
    \|\rho - \brho\|_{H^1}^2 \|\rho -\brho\|_{L^2} %\right) 
    \\
    \leq  2 C_S'C_SC_{ell}   \sqrt{y_1(t)} (y_1(t) + y_2(t) ).
\end{multline}
We have used Young's inequality in the last step. This leads us to:
\begin{equation}\label{est:y3}
  \int_0^{T'} y_3(t) dt \leq 2 C_S'C_SC_{ell} \sup_t \sqrt{y_1(t)}  \int_0^{T'} y_1(t) + y_2(t) dt .
\end{equation}

Equations \eqref{near:gron} and \eqref{est:y3} show that our analysis fits into the framework of the generalized Grönwall lemma (see Lemma \ref{prp:GeneralizedGronwall})
with $y_1,y_2,y_3,$ and $a$ as above, and $B:= 2 C_S'C_SC_{ell}$, $\beta := \frac 12$, and
\begin{equation}\label{eq:AE}
    A:= y_1(0)  + \int_0^{T} \| R_\rho\|_{H^{-1}}^2  dt,
    \quad
    \text{and}
    \quad E := \exp\left(  \int_0^{T}  a(t) dt\right).
\end{equation}

\begin{lemma} \cite[Proposition 6.2]{Bartels} \label{prp:GeneralizedGronwall}
    Suppose that nonnegative functions $y_1 \in C([0, T]), y_2, y_3 \in L^1([0, T]), a \in L^{\infty}([0, T])$, and a real number $A \geq 0$ satisfy
    \begin{equation*}
    y_1\left(T^{\prime}\right)+\int_0^{T^{\prime}} y_2(t) \mathrm{d} t \leq A+\int_0^{T^{\prime}} a(t) y_1(t) \mathrm{d} t+\int_0^{T^{\prime}} y_3(t) \mathrm{d} t
    \end{equation*}
    for all $T^{\prime} \in[0, T]$ and that, in addition, for $B \geq 0, \beta>0$, and every $T^{\prime} \in[0, T]$
    \begin{equation*}
    \int_0^{T^{\prime}} y_3(t) \mathrm{d} t \leq B\left(\sup _{t \in\left[0, T^{\prime}\right]} y_1^\beta(t)\right) \int_0^{T^{\prime}}\left(y_1(t)+y_2(t)\right) \mathrm{d} t .
    \end{equation*}
    Let $E:=\exp \left(\int_0^T a(t) \mathrm{d} t\right)$. Then, provided $8 A E \leq(8 B(1+T) E)^{-1 / \beta}$ is satisfied the following inequality holds:
    \begin{equation*}
        \sup _{t \in[0, T]} y_1(t)+\int_0^T y_2(t) \mathrm{d} t \leq 8 A \exp \left(\int_0^T a(s) \mathrm{d} s\right) .
    \end{equation*}
\end{lemma}
Thus, we have the following conditional stability result:
\begin{theorem} [Conditional stability estimate]
    \label{thm:stabilityest} 
    For a fixed $T > 0$, let $(\rho, c)$ be a weak solution to \eqref{eq:KS1}--\eqref{eq:ic} and let $(\bar{\rho}, \bar{c})$ be a strong solution to \eqref{eq:KS-p1}--\eqref{eq:bc-p2}. Define $y_1, y_2, y_3,$ and $a$ as in \eqref{eq:y123a}. Let $B = 2 C_S' C_S C_{\text{ell}}$, $\beta = \frac{1}{2}$, and $A, E$ as in \eqref{eq:AE}. Here, $C_S$ and $C_S'$ are constants from the embeddings $H^1 \rightarrow L^6$ and $H^1 \rightarrow L^3$, respectively, while $C_{\text{ell}}$ is the constant of elliptic regularity from \eqref{assumption:elliptic}.
    Then, provided the condition
    \begin{equation} \label{eq:estcondition}
        8 A E ( 8B (1 +T) E)^{2} \leq 1
    \end{equation}
    is satisfied, the difference $\rho - \brho$ is controlled as follows:
    \begin{multline}\label{mainestimate}
    \sup_{t \in [0,T]}   \|  \rho (t,\cdot) - \brho (t,\cdot)\|_{L^2(\Omega)}^2  +  \int_0^T |  \rho (t,\cdot) - \brho (t,\cdot)|_{H^1(\Omega)}^2 dt
    \\
    \leq 8 \left(    \|  \rho (0,\cdot) - \brho (0,\cdot)\|_{L^2(\Omega)}^2 + \int_0^{T} \| R_\rho\|_{H^{-1}(\Omega)}^2  dt   \right) \\
    \times \exp \left( \int_0^T 3 C_S^2C_{ell}^2 \|\brho(t,\cdot)\|_{L^3(\Omega)}^2  + 3 \|\nabla \bc(t,\cdot)\|_{L^\infty(\Omega)}^2 + \frac 1 3 dt\right).
\end{multline}
\end{theorem}

\begin{remark}[Stability framework] 
    The above stability framework is independent of the way the approximate solution $(\bar \rho, \bar c)$ is obtained. 
    We apply this framework to a posteriori error analysis of a discontinuous Galerkin scheme in Section \ref{sec:dgscheme}. It should be noted that, in principle, this framework is applicable to approximations obtained from any numerical scheme, e.g., some finite volume scheme \cite{GiesselmannKolbe}. The major step that needs to be done is to define a suitable reconstruction, i.e. one that is sufficiently regular. Moreover, it is desirable that the corresponding residual scales optimally.
\end{remark}

One advantage of using the estimate \eqref{mainestimate} is that
the assumption of the existence of a weak solution $(\rho, c)$  until time $T$ can be verified \emph{a posteriori}.
Indeed, the condition \eqref{eq:estcondition} is verifiable a posteriori, and if it holds, the error estimate itself rules out blow-up before time $T$, thus proving a posteriori that a sufficiently regular weak solution exists. More precisely:

\begin{corollary} [A posteriori verifiable regularity]
    \label{thm:guarantee_existence}
    Let $\rho_0 \in L^2(\Omega)$ and 
    let $T_{max} > 0$ be the maximal existence time of the weak solution $\rho$ to \eqref{eq:KS1}--\eqref{eq:ic}.
    Suppose that there exists an approximate strong solution $\brho$ of \eqref{eq:KS-p1}--\eqref{eq:bc-p2} such that at some time $T>0$
   the condition \eqref{eq:estcondition} is satisfied, $\|\bar \rho\|_{L^\infty(0,T; L^2(\Omega))}$ is finite and the right hand side of \eqref{mainestimate} is finite. Then $T_{max} > T$, i.e., the weak solution exists at least until time  $T$. 
\end{corollary}

\begin{proof}
    If \eqref{eq:estcondition}  holds for some time $T>0$ then it also holds for all earlier times, since the left hand side of \eqref{eq:estcondition} is increasing in time.  Suppose $T_{max} \leq T$. Then,  \eqref{mainestimate} evaluated at time $T$ provides a uniform upper bound for $\|\rho(t, \cdot)\|_{L^2(\Omega)}$ for almost all  $t \in [0,T]$, since $\left\| \bar{\rho}(t, \cdot) \right\|_{L^2(\Omega)}$ is finite owing to the assumption.
    In contrast, Remark \ref{rmk:uniformbound} and Lemma \ref{rmk:blow-up} imply that the $L^2$-norm of $\rho$ must blow up  as $t \nearrow T_{max}$, since $d \leq 3$. This is a contradiction.
\end{proof}

\begin{remark} [A posteriori verifiability] \label{rmk:aposteri_verifiability}
    Given a numerical solution at time $T$, if its reconstruction (see Section \ref{sec:reconstruction}) has a finite $L^\infty(L^2)$-norm and satisfies the condition \eqref{eq:estcondition}, which depends on the data $\rho_0$, then the exact solution cannot have blown up before time $T$.
    However, the condition \eqref{eq:estcondition} contains the exponential $E$, making it quite restrictive. 
    $E$ is related to error propagation. Due to the nonlinear and somewhat unstable dynamics of the Keller-Segel model and the fact that it  $E$ is potentially large. In particular, there is no way to construct a reasonable numerical scheme that ensures that the norms contained in
    $E$ are small. Indeed, if an approximating solution $\bar{\rho}$ is close to the exact solution $\rho$, then $E$ basically reflects properties of the exact solution. This situation is analogous to the situation encountered in \cite{2023Giesselmann}, \cite{Cangiani_2016}, and \cite{2011Bartels}.
    However, using a high order method can mitigate this to a certain extent. 
    Roughly speaking, $A$ goes to zero with powers of $\Delta x$ and $\Delta t$ that depend on the order of the method and the smoothness of the solution.
    %; in contrast, the size of $E$ grows in time and depends on certain norms of the numerical solution - as long as we use a stable scheme this term basically reflects properties of the exact solution.
\end{remark}
% In the context of Remark \ref{rmk:aposteri_verifiability}, we present a discontinuous Galerkin scheme for the Keller-Segel system in the next section, which offers advantages due to its ability to allow for high-order approximations and $hp$-adaptivity.

We now apply the stability framework of Theorem \ref{thm:stabilityest} to a specific scheme and a specific reconstruction. In the subsequent section, we provide a Discontinuous Galerkin scheme, which is attractive since it allows for high-order approximations and hp-adaptivity.

\section{Discontinuous Galerkin scheme}% for the Keller-Segel Model}
\label{sec:dgscheme}
We employ a discontinuous Galerkin (dG) finite element method for spatial discretisation.
To ensure our paper is self-contained, we start this section by briefly recalling some basic notations. Details can be found in, e.g., \cite{Di_Pietro_2012}.
\begin{definition}[Finite element space] \label{definition:finite-element-space}
    Let $\Omega \subset R^d$ with $d=2,3$ be decomposed into a set $\mathcal{T}$, called a \emph{mesh}, of disjoint polyhedra $\{T \}$. Each $T \in \mathcal{T}$ is called a \emph{mesh element}.
    Let $h_T$ denote the \emph{diameter} of $T$ and the \emph{mesh size} is defined as $h:=\max_{T \in \mathcal{T}}h_T$.
    We use the notation $\mathcal{T}_h$ for $\mathcal{T}$ with mesh size $h$.

    We define $\mathcal{F}_h^i$ as the set of common \emph{interfaces} of cells in the mesh $\mathcal{T}_h$, and $\mathcal{F}_h^b$ as the set of \emph{boundary faces} on the boundary $\partial \Omega$. We set
    \begin{equation*}
    \mathcal{F}_h := \mathcal{F}_h^i \cup \mathcal{F}_h^b
    \end{equation*}
    as the set of all faces in the mesh. For all $F \in \mathcal{F}_h$, 
    % in dimension $d \geq 2$, 
    we set $h_F$ to be equal to the diameter of the face $F$. Furthermore, for any $T \in \mathcal{T}_h$, we set
    \begin{equation} \nonumber
        \mathcal{F}_T := \left\{F \in \mathcal{F}_h \mid F \subset \partial T\right\},
    \end{equation}
    which collects the mesh faces comprising the boundary of $T$.

    We say that a mesh sequence $\left\{ \mathcal{T}_h \right\}_{h>0}$ is \emph{admissible} if it is shape- and contact-regular and if it has optimal polynomial approximation properties. Definitions of these notions can be found in \cite[Section 1.4]{Di_Pietro_2012}.
    Throughout this paper, we assume that the mesh sequence $\left\{ \mathcal{T}_h \right\}_{h>0}$ is \emph{admissible}.

    Let $\mathbb{P}_d^k$ denote the space of polynomials in space dimension $d$ of degree less than or equal to $k$. We define the \emph{discontinuous finite element space} as
    \begin{equation} \nonumber
        V_h :=
        \{v : \Omega \to \mathbb{R} \, : \,
        v | _{T} \in \mathbb{P}_d^k{(T)} \ \forall\, T \in \mathcal{T}_h\}
    \end{equation}
    with polynomial degree $k \geq 1$.
\end{definition}

\begin{remark}
    For our numerical analysis and implementation, we assume that the domain $\Omega$ is a polyhedron in $\mathbb{R}^d$ to make the presentation more concise. Thus the decomposition $\Omega = \cup \mathcal{T}_h$ in Definition \ref{definition:finite-element-space} makes sense. For domains with curved boundaries, see \cite[Assumption 1.7]{Di_Pietro_2012}.
\end{remark}

\begin{definition}[Broken Sobolev spaces and broken gradients]\label{definition:broken-sobolev-space}
    We introduce the \emph{broken Sobolev space}
    \begin{equation} \nonumber
        \begin{aligned}
            W^{m, p}\left(\mathcal{T}_h\right)&:=\left\{v \in L^p(\Omega) \, : \, v|_T \in W^{m, p}(T) \ \forall T \in \mathcal{T}_h \right\}
        \end{aligned}
    \end{equation}
    where $m \in \mathbb{N}\cup\{0\}$ and $p \geq 1$.
    For $p=2$, we set $H^m(\mathcal{T}_h): = W^{m, 2}\left(\mathcal{T}_h\right)$.
    The \emph{broken gradient} $\nabla_h: W^{1, p}\left(\mathcal{T}_h\right) \rightarrow$ $\left[L^p(\Omega)\right]^d$ is defined by
    \begin{equation*}
   \left.\quad\left(\nabla_h v\right)\right|_T:=\nabla\left(\left.v\right|_T\right) \quad \forall v \in W^{1, p}\left(\mathcal{T}_h\right), \  \forall T \in \mathcal{T}_h.
    \end{equation*}
\end{definition}

%\begin{definition}[Broken gradient] 
%    The \emph{broken gradient} $\nabla_h: W^{1, p}\left(\mathcal{T}_h\right) \rightarrow$ $\left[L^p(\Omega)\right]^d$ is defined such that, for all $v \in W^{1, p}\left(\mathcal{T}_h\right)$,
%    \begin{equation*}
 %   \forall T \in \mathcal{T}_h,\left.\quad\left(\nabla_h v\right)\right|_T:=\nabla\left(\left.v\right|_T\right).
 %   \end{equation*}
%\end{definition}

\begin{definition} 
    For $v_h \in V_h$, we define the operators $v_h^+$ and $v_h^-$ on $\mathcal{F}_h^i$ as follows:
    For any $F \in \mathcal{F}_h^i$ and a unit normal vector $\mathrm{n}_F$ on $F$, we set
    \begin{equation} \nonumber
        \begin{aligned}
            v_h^+(x)&:= \lim_{\delta \to 0^+} v_h(x + \delta \mathrm{n}_F)\\
            v_h^-(x)&:= \lim_{\delta \to 0^+} v_h(x - \delta \mathrm{n}_F)
        \end{aligned}
    \end{equation}
    for a.e. $x \in F$.
\end{definition}

\begin{definition} [Averages and jump]
    Let $\omega_+, \omega_- \in L^{\infty}(\mathcal{F}_h^i)$ be a function satisfying $\omega_+ + \omega_- = 1$ and bounded from below by a positive real number. For $v_h \in V_h$, we define its \emph{weighted average} and \emph{jump} operators on $\mathcal{F}_h^i$ as 
    \begin{equation}
        \begin{aligned}
            \dc{v_h}_\omega&:=\omega_+v_h^+ + \omega_-v_h^-\\
            \llbracket v_h \rrbracket&:=v_h^+ - v_h^-
        \end{aligned}
    \end{equation}
    In particular, if $\omega_+ = \omega_- = \frac{1}{2}$ then
    we denote $\dc{v_h}_\omega$ as $\dc{v_h}$ and call it the \emph{average} operator, i.e.,
    \begin{equation}
        \begin{aligned}
            \dc{v_h}&:=\frac{1}{2}\left(v_h^+ + v_h^-\right).
        \end{aligned}
    \end{equation}
    Moreover, we introduce the \emph{skew-weighted average} operator of $v_h$ on $\mathcal{F}_h^i$ as
    \begin{equation}
        \begin{aligned}
            \dc{v_h}_{\bar{\omega}}&:=\omega_-v_h^+ + \omega_+v_h^-.
        \end{aligned}
    \end{equation}
    When $v_h$ is vector-valued, these operators act componentwise.
\end{definition}
\begin{remark} [Sign of jump in normal directions]
    Note that the sign of $\llbracket v_h \rrbracket \mathrm{n}_F$ is definite on each face $F \in \mathcal{F}_h^i$, since it is independent of the choice of the normal vector $\mathrm{n}_F$.
\end{remark}

Using this notation, we define the \emph{symmetric interior penalty} (SIP) bilinear form:
\begin{equation}\label{eq:sip}
    \begin{aligned}
        & a^{sip}_h: V_h \times V_h \rightarrow  \mathbb{R}\\
        & a^{sip}_h(u_h,\phi_h) 
         : = 
          \sum_{T \in \mathcal{T}_h}\int_{T}
          \nabla u_h \cdot \nabla \phi_h\\
         & 
          -
          \sum_{F \in \mathcal{F}_h^i}\int_{F} \Big( \jump{u_h}\dc{\nabla \phi_h}\cdot \mathrm{n}_F  + \jump{\phi_h}\dc{\nabla u_h}\cdot \mathrm{n}_F
          -
          \frac{\eta}{h_F} \jump{u_h}\jump{\phi_h}\Big),
    \end{aligned}
\end{equation}
where $\eta>0$ is a parameter which is chosen sufficiently large that $a^{sip}_h$ is semi-positive definite.

We discretise the chemotaxis term using the \emph{weighted symmetric interior penalty} (wSIP) form, which is bilinear with respect to the second and third arguments:
\begin{equation} \label{eq:wsip}
    \begin{aligned}
        & a^{wsip}_h: V_h^+ \times V_h \times V_h \rightarrow  \mathbb{R} \\
        & a^{wsip}_h(v_h;u_h, \psi_h) = \sum_{T \in \mathcal{T}_h}\int_T \left( v_h \nabla_h u_h \right)\cdot\nabla\psi_h \\
        & - \sum_{F \in \mathcal{F}_h^i} \int_F \left(  \llbracket u_h \rrbracket \dc{v_h \nabla_h \psi_h}_{\omega} \cdot \mathrm{n}_F  + \llbracket \psi_h \rrbracket\dc{v_h \nabla_h u_h}_{\omega}\cdot \mathrm{n}_F - \sigma \frac{\gamma_{v_h}}{h_F} \llbracket u_h \rrbracket \llbracket \psi_h \rrbracket \right),
    \end{aligned}
\end{equation}
where the weights $\omega_+$ $\omega_-$ are chosen as
\begin{equation*}
    \omega_+ = \frac{v^-_h}{v^+_h + v^-_h}, \ \omega_- = \frac{v^+_h}{v^+_h + v^-_h},
\end{equation*} $\sigma>0$ is a constant sufficiently large so that $a_h^{wsip}$ is semi-positive definite, and the diffusion-dependent penalty parameter $\gamma_{v_h}$ is set to
\begin{equation*}
    \gamma_{v_h} := \frac{2v_h^+v_h^-}{v_h^+ + v_h^-}.
\end{equation*}
Here $V_h^+ := \{v_h \in V_h \, : \exists c > 0 \text{ s.t. } \, v_h(x) \geq c > 0 \text{ for all } x \in \Omega \}$.

% We consider the following DG semi-discretisation:
% \begin{definition}
% Given $\rho_h^0 \in V_h$ we define  $\left( \rho_h, c_h \right) \in C^1(0,T;V_h)\times C(0,T;V_h)$ as the solution of
% \begin{equation}\label{def:dg}
%     \begin{aligned}
%         \int_\Omega \left( \partial_t \rho_h \right)\phi_h &= -a^{sip}_h\left( \rho_h, \phi_h \right) + a^{wsip}_h\left( \rho_h; c_h, \phi_h \right)  \quad \text{for all} \quad \phi_h \in V_h, \\
%         a^{sip}_h\left( c_h, \psi_h \right) + \int_\Omega c_h \psi_h &= \int_\Omega \rho_h \psi_h   \quad\quad\quad\quad\quad\quad\quad\quad\quad\quad\quad \ \text{for all}    \quad \psi_h \in V_h,
%     \end{aligned}
% \end{equation}
% with $\rho_h(0)=\rho_h^0$.
% \end{definition}

\begin{remark}[Weakly enforced boundary conditions]
    Neither the symmetry nor the penalty terms of equations \eqref{eq:sip} and \eqref{eq:wsip} contain contributions from boundary faces, since we consider the homogeneous Neumann boundary condition. See \cite[Section 4.2.2]{Di_Pietro_2012} for details.
\end{remark}

Let $\pi_h$ be the $L^2$-orthogonal projection onto $V_h$. That is, it satisfies $(\pi_h v, \phi_h)_{L^2(\Omega)} = (v, \phi_h)_{L^2(\Omega)}$ for all $v \in L^2(\Omega)$ and $\phi_h \in V_h$. 
Define $\rho_h^0:=\pi_h \rho_0$.
Our numerical scheme for the Keller-Segel system \eqref{eq:KS1}--\eqref{eq:ic} is then formulated as follows:

Find $\rho_h^{n+1}$ and $c_h^{n+1}$ by solving the following linear system for each time step $n =0, \dots, N-1$:
\begin{equation} \label{eq:fully_discrete}
    \begin{aligned}
        \left( \frac{\rho_h^{n+1} - \rho_h^{n}}{\tau_n}, \phi_h \right)_{L^2(\Omega)} + a_h^{sip}(\rho_h^{n+1}, \phi_h) - a_h^{wsip}(\rho_h^n; c_h^{n+1}, \phi_h)&=0 \quad \forall \quad \phi_h \in V_h\\
        a_h^{sip}(c_h^{n+1}, \psi_h) + \left( c_h^{n+1}, \psi_h \right)_{L^2(\Omega)} - \left( \rho_h^{n+1}, \psi_h \right)_{L^2(\Omega)} &=0 \quad \forall \quad \psi_h \in V_h
    \end{aligned}
\end{equation}
where $0 = t_0 < t_1 < \cdots < t_N = T$ and $\tau_n:=t_{n+1} - t_n$.

\section{Space-time reconstruction} \label{sec:reconstruction}
In this section, we define the space-time reconstruction of the numerical solution $\left\{ (\rho_h^n, c_h^n) \right\}_{n=0}^N$ obtained from \eqref{eq:fully_discrete}. To utilize our stability framework, reconstruction is necessary since inserting the numerical solution $\left\{ (\rho_{h}^n, c_{h}^n) \right\}_{n=0}^N$ into \eqref{eq:KS1}--\eqref{eq:KS2} does not make sense and would result in singular residuals. 

Let us first define the \emph{time reconstruction} $\bar{\rho}^{\tau}:[0,T] \rightarrow V_h$ of $\left\{ \rho_h^n \right\}_{n=0}^N$ as the linear interpolant at temporal points $\{t_n\}_{n=0}^N$. That is, for each $n=0,\cdots, N-1$ we define
\begin{equation} \label{eq:time_reconstruction}
    \bar{\rho}^{\tau}(s) := \frac{t_{n+1} - s}{t_{n+1} - t_n} \rho_h^n + \frac{s - t_n}{t_{n+1} - t_n} \rho_h^{n+1} \quad \quad \text{for } s \in I^{n}:=[t_n, t_{n+1}].
\end{equation}
The reconstruction $\bar{c}^{\tau}:[0,T] \rightarrow V_h$ of $\left\{ c_h^n \right\}_{n=0}^N$ is defined analogously. Note that $\bar{\rho}^{\tau}$ and $\bar{c}^{\tau}$ are globally continuous in time.
Then
\begin{equation} \nonumber
    \begin{aligned}
        \left( \partial_t \bar{\rho}^{\tau} , \phi_h \right)_{L^2(\Omega)} &= \left( \frac{\rho_h^{n+1} - \rho_h^{n}}{\tau_n}, \phi_h \right)_{L^2(\Omega)} \\ 
        &= - a_h^{sip}(\rho_h^{n+1}, \phi_h) + a_h^{wsip}(\rho_h^n; c_h^{n+1}, \phi_h) \quad \forall \quad \phi_h \in V_h.
    \end{aligned}
\end{equation}
Thus, we define the \emph{temporal residual} $R^{\tau} \in V_h$ to satisfy for all $\phi_h \in V_h$: 
\begin{equation}
    \begin{aligned}
        &\left( R^{\tau}, \phi_h \right)_{L^2(\Omega)} \\
        &:= \left( \partial_t \bar{\rho}^{\tau}, \phi_h \right)_{L^2(\Omega)} + a_h^{sip}(\bar{\rho}^{\tau}, \phi_h) - a_h^{wsip}(\bar{\rho}^{\tau}; \bar{c}^{\tau}, \phi_h) \\
        &\ = a_h^{sip}(\bar{\rho}^{\tau}, \phi_h) - a_h^{wsip}(\bar{\rho}^{\tau}; \bar{c}^{\tau}, \phi_h) - a_h^{sip}(\rho_h^{n+1}, \phi_h) + a_h^{wsip}(\rho_h^n; c_h^{n+1}, \phi_h) \\
        % &\quad \forall \quad \phi_h \in V_h . 
        \label{eq:temporal_residual}
    \end{aligned}
\end{equation}
We expect that the residual $R^{\tau}$ is of order $O(\tau)$ where $\tau:=\max_{n}\tau_n$.

Now we define spatial reconstructions of the (pointwise) dG solutions $\left( \bar{\rho}^{\tau}, \bar{c}^{\tau} \right)$ by employing so-called  \emph{elliptic reconstruction} \cite{Makridakis2003,Makridakis2007}, which is a well-known method that leads to residuals that have optimal order in several other problems. The residual serves as a perturbation of the PDE with optimal order. We will make use of this fact to argue that our estimator is of optimal order.

For notational convenience, we use the abbreviations
\begin{equation} \nonumber
    \left( \cdot, \cdot \right) := \left( \cdot, \cdot \right)_{L^2(\Omega)} \quad \text{and} \quad \left\langle \cdot, \cdot \right\rangle := \left\langle \cdot, \cdot \right\rangle_{H^{-1}(\Omega), H^1(\Omega)}.
\end{equation} 
For the sake of brevity, we will also omit the time dependency of functions when the context makes it clear.
Let $A_h:V_h \longrightarrow V_h$ be the following discrete version of the operator $-\Delta$:
\begin{equation} \label{def:linearform}
     \left( A_hv_h, w_h \right) := a^{sip}_h\left( v_h, w_h \right) \quad \quad \text{for all } v_h, w_h \in V_h.
\end{equation}
For any fixed time $t \in [0,T]$, we define the \emph{space-time reconstruction} $\bar{\rho}^{\xi}(t, \cdot) \in H^1(\Omega)$ of $\bar{\rho}^{\tau}(t, \cdot)$ as the solution of the elliptic problem
\begin{equation} \label{eq:reconstruct}
    \begin{aligned}
        -\Delta\bar{\rho}^{\xi} &= A_h\bar{\rho}^{\tau}  \ \quad \text{in } \Omega  \\ 
       \nabla \bar{\rho}^{\xi}\cdot \mathrm{n} &= 0  \quad \quad \quad \text{on } \partial \Omega
    \end{aligned}
\end{equation}
such that $\bar{\rho}^{\xi}\left( t, \cdot \right)$ has the same mean value as the discrete solution, i.e.,
\begin{equation} \label{eq:reconstruct_mean}
    \int_\Omega \bar{\rho}^{\xi}\left( t, \cdot \right) dx = \int_\Omega \bar{\rho}^{\tau}(t, \cdot) dx.
\end{equation}
More precisely, $\bar{\rho}^{\xi}(t, \cdot)$ satisfies the following:
\begin{equation} \nonumber
    a(\bar{\rho}^{\xi}(t), \phi):=(\nabla\bar{\rho}^{\xi}(t), \nabla \phi)= \left( A_h\bar{\rho}^{\tau}(t), \phi \right) \quad \text{for all} \ \phi \in H^1(\Omega).
\end{equation} 
%where $a(\cdot, \cdot)$ is defined by $a(u,v):= \int \nabla u \nabla v dx$. 
Note that $\bar{\rho}^{\tau}(t, \cdot)$ is the SIP-dG solution to \eqref{eq:reconstruct}-\eqref{eq:reconstruct_mean} due to the definition \eqref{def:linearform}.
Similarly, for a fixed time $t \in [0,T]$, let us define the reconstruction $\tilde{c}^{\xi}(t, \cdot) \in H^1(\Omega)$ of $\bar{c}^{\tau}(t, \cdot)$ as the solution of the elliptic problem 
\begin{equation} \label{eq:reconstruct2}
    \begin{aligned}
        -\Delta\tilde{c}^{\xi} + \tilde{c}^{\xi} &= \bar{\rho}^{\tau}  \ \quad \text{in } \Omega  \\ 
        \nabla \tilde{c}^{\xi}\cdot \mathrm{n} &= 0  \quad \quad \text{on } \partial \Omega.
    \end{aligned}
\end{equation}
That is, it satisfies 
\begin{equation} \nonumber
     a( \tilde{c}^{\xi}(t), \psi ) + \left( \tilde{c}^{\xi}(t), \psi \right)  = ( \bar{\rho}^{\tau}(t), \psi ) \quad \text{for all }\psi \in H^1(\Omega).
\end{equation}
Finally, we define $\bar{c}^{\xi}(t) \in H^1(\Omega)$ as the solution to the elliptic problem:
\begin{equation} \label{eq:reconstruct3}
    \begin{aligned}
        -\Delta\bar{c}^{\xi} + \bar{c}^{\xi} &= \bar{\rho}^{\xi}  \ \quad \text{in } \Omega  \\ 
        \nabla \bar{c}^{\xi}\cdot \mathrm{n} &= 0  \quad \quad \text{on } \partial \Omega.
    \end{aligned}
\end{equation}
In other words, it satisfies
\begin{equation} \nonumber
     a( \bar{c}^{\xi}(t), \chi ) + \left( \bar{c}^{\xi}(t), \chi \right)  = ( \bar{\rho}^{\xi}(t), \chi ) \quad \text{for all }\chi \in H^1(\Omega).
\end{equation}

\begin{remark} [Regularity and continuity of reconstructions] \label{remark:regularity}
    %It is important to mention that, based on Assumption \ref{assumption:elliptic}, 
    For any fixed time $t \in [0, T]$, we have
    \begin{equation} \nonumber
        \bar{\rho}^{\xi}(t, \cdot) \in H^2(\Omega), \quad \tilde{c}^{\xi}(t, \cdot) \in H^2(\Omega),
        %  \quad \bar{c}^{\xi}(t, \cdot) \in H^2(\Omega),
    \end{equation}
    owing to elliptic regularity. Additionally, using elliptic regularity and the fact $\bar{\rho}^{\xi} \in H^2(\Omega)$,  we obtain 
    \begin{equation} \nonumber
        \bar{c}^{\xi}(t, \cdot) \in H^4(\Omega).
    \end{equation}
    The continuity (in time) of $\bar{\rho}^{\tau}$ implies that $\bar{\rho}^{\xi}$, $\tilde{c}^{\xi}$, and $\bar{c}^{\xi}$ are also continuous in time.
    Thus we have
    \begin{equation} \nonumber
        \bar{\rho}^{\xi} \in C( [0,T]; H^2(\Omega)), \quad \tilde{c}^{\xi} \in C ( [0,T]; H^2(\Omega) ), \quad \bar{c}^{\xi} \in C ( [0,T]; H^4(\Omega) ) .
    \end{equation}
    Since the bilinear form $a(\cdot, \cdot)$ is time-independent, it follows that
    \begin{equation} \nonumber
        a(\partial_t\bar{\rho}^{\xi}, \phi) = \langle \partial_t A_h\bar{\rho}^{\tau}(t), \phi \rangle
    \end{equation} 
    for all $\phi \in H^1(\Omega)$. 
\end{remark}

\begin{remark} [a posteriori control for elliptic problems]
    Reliable and efficient a posteriori error estimators for the $L^2$-norm error associated with continuous finite element discretisations of Poisson's equation are detailed in \cite[ (4.4), (4.5)]{Makridakis2003}. These estimators can be adapted to SIP-dG schemes in a straightforward manner.
    Moreover, \cite[Theorem 3.1]{Karakashian_2003} provides reliable and efficient a posteriori error estimators for the dG-norm error arising from interior penalty dG discretisations of Poisson's equation.
    We also introduce an a posteriori error estimator in terms of the $H^{-1}$-norm.
    To state all these estimators concisely, we define
    \[ R_T [\bar{\rho}^{\tau}, f]:= \left\| \Delta \bar{\rho}^{\tau} + f \right\|_{L^2(T)}^2, \ 
     R_F^1 [\bar{\rho}^{\tau}]:= \left\| \llbracket \nabla_h \bar{\rho}^{\tau} \rrbracket \cdot \mathrm{n}_F \right\|_{L^2(F)}^2, \ R_F^0 [\bar{\rho}^{\tau}]:= \left\| \llbracket \bar{\rho}^{\tau} \rrbracket \right\|_{L^2(F)}^2.
    \]
    Then we have the following error estimates:
    \begin{eqnarray}
        \left\| \bar{\rho}^{\tau} - \bar{\rho}^{\xi}  \right\|_{L^2(\Omega)} &\leq& C_0 \operatorname{E}_0\left[ \bar{\rho}^{\tau}, A_h\bar{\rho}^{\tau} \right]\label{rhoest1}\\
         \| \bar{\rho}^{\tau} - \bar{\rho}^{\xi}\|_{\mathrm{dG}} &\leq& C_1 \operatorname{E}_1[\bar{\rho}^{\tau}, A_h\bar{\rho}^{\tau}]\label{rhoxest1}\\
        \left\| \partial_t \bar{\rho}^{\tau} - \partial_t \bar{\rho}^{\xi} \right\|_{H^{-1}(\Omega)} &\leq& C_{-1} \operatorname{E}_{-1}\left[ \partial_t \bar{\rho}^{\tau}, \partial_t A_h\bar{\rho}^{\tau} \right]\label{rhotest1}
    \end{eqnarray}
    where the dG-norm is defined as
        \begin{equation} \nonumber
        \|v\|_{\mathrm{dG}}^2 :=\left\|\nabla_h v\right\|_{\left[L^2(\Omega)\right]^d}^2+\sum_{F \in \mathcal{F}_h} \frac{1}{h_F}\|\llbracket v \rrbracket \|_{L^2(F)}^2,
    \end{equation}
    and the elliptic error estimators are defined as, for $k\geq1$,
    \begin{eqnarray} \label{rhoest2}
       \operatorname{E}_0\left[ \bar{\rho}^{\tau}, f_h \right]^2 := \sum_{T \in \mathcal{T}_h} h^4_T R_T [\bar{\rho}^{\tau}, f_h] 
       + \sum_{F \in \mathcal{F}_h^i} h_F^3 R_F^1 [\bar{\rho}^{\tau}] + \eta^2 \sum_{F \in \mathcal{F}_h^i} h_F R_F^0 [\bar{\rho}^{\tau}],
       \\
      \operatorname{E}_1\left[ \bar{\rho}^{\tau}, f_h \right]^2 :=\sum_{T \in \mathcal{T}_h} h^2_T R_T [\bar{\rho}^{\tau}, f_h] 
       + \sum_{F \in \mathcal{F}_h^i} h_F R_F^1 [\bar{\rho}^{\tau}] + \eta^2 \sum_{F \in \mathcal{F}_h^i} h_F^{-1} R_F^0 [\bar{\rho}^{\tau}],
    \end{eqnarray}
    and for $k \geq 2 $,
    \begin{equation}
        \label{rhotest2}
             \operatorname{E}_{-1}\left[ \bar{\rho}^{\tau}, f_h \right]^2 := \sum_{T \in \mathcal{T}_h} h^6_T R_T [\bar{\rho}^{\tau}, f_h] 
             + \sum_{F \in \mathcal{F}_h^i} h_F^5 R_F^1 [\bar{\rho}^{\tau}] + \eta^2 \sum_{F \in \mathcal{F}_h^i} h_F^{3} R_F^0 [\bar{\rho}^{\tau}].
    \end{equation}
    % {\tt Kiwoong: For $k=1$, the error estimator $\operatorname{E}_{-1}$ is not defined. Indeed, in the implementation for $k=1$ we have used $\operatorname{E}_{0}$ instead. We should make a comment on this somewhere.}
    The real numbers $C_0$, $C_1$, and $C_{-1}$ are positive constants dependent only on the regularity of $\Omega$. They can be evaluated on a convex polygonal domain, see \cite[Theorem 5.45]{Di_Pietro_2012}
    
    Analogously, we have
    \[  \| \bar{c}^{\tau} - \tilde{c}^{\xi}\|_{\operatorname{dG}} \leq \tilde{C}_1 \operatorname{\tilde{E}}_1[\bar{c}^{\tau}, \bar{\rho}^{\tau}]\]
     with
    \begin{multline}\label{cest} 
            \operatorname{\tilde{E}}_1[\bar{c}^{\tau}, f_h]^2
            := 
            \sum_{T\in\grid_h}
            h_T^2 \|f_h - \bar{c}^{\tau} + \Delta \bar{c}^{\tau}\|_{L^2(T)}^2\\
           +
            \sum_{F\in \mathcal{F}_h} h_F \| \jump{\nabla \bar{c}^{\tau}}\cdot \mathrm{n}_F \|_{L^2(F)}^2
            +
            \sigma^2 \sum_{F\in \mathcal{F}_h} h_F^{-1} \|\jump{\bar{c}^{\tau}}\|_{L^2(F)}^2.
    \end{multline}
   The positive real number $\tilde{C}_1$ also can be evaluated as previously mentioned.
   
   Finally, $\| \tilde{c}^{\xi} - \bar{c}^{\xi}\|_{H^1(\Omega)}$ is controlled by $\| \bar{\rho}^{\tau} - \bar{\rho}^{\xi}\|_{L_2(\Omega)},$
   which is controlled by $\operatorname{E}_0[\bar{\rho}^{\tau}, A_h\bar{\rho}^{\tau}]$.
   Thus we have
   \begin{equation} \label{eq:cest2}
     \| \bar{c}^{\tau} - \bar{c}^{\xi}\|_{\operatorname{dG}} \leq C_{ell}C_0\operatorname{E}_0[\bar{\rho}^{\tau}, A_h\bar{\rho}^{\tau}] + \tilde{C}_1\operatorname{\tilde{E}}_1[\bar{c}^{\tau}, \bar{\rho}^{\tau}]. 
   \end{equation}

   For all elliptic error estimators and generic functions $u_h \in V_h$ we subsequently write $E_i[u_h]:= E_i[u_h, A_h u_h]$ for brevity. 
    
\end{remark}

\noindent
For any fixed $t \in [0, T]$, we define the \emph{space-time residual} $R^{\xi}(t, \cdot) \in H^{-1}(\Omega)$ as 
\begin{equation} \label{eq:residualpair}
    \left\langle R^{\xi}, \phi \right\rangle %_{H^{-1}, H^1}
    := \left\langle \partial_t\bar{\rho}^{\xi}, \phi \right\rangle%_{H^{-1}, H^1} 
    - \left( \bar{\rho}^{\xi}\nabla \bar{c}^{\xi} - \nabla\bar{\rho}^{\xi}, \nabla \phi \right)
\end{equation}
for all $\phi \in H^1(\Omega)$. The regularities mentioned in Remark \ref{remark:regularity} ensure that the above definition is well-defined, as long as $d \leq 3$, and lead to $R^{\xi} \in L^2(0,T; H^{-1}(\Omega))$.
This allows us to interpret $(\bar{\rho}^{\xi}, \bar{c}^{\xi})$ as a strong solution to the perturbed problem:
\begin{equation} \label{eq:strong_formulation}
    \begin{matrix}
        \begin{aligned}
            \partial_t\bar{\rho}^{\xi} +\nabla\cdot(\bar{\rho}^{\xi}\nabla \bar{c}^{\xi} - \nabla \bar{\rho}^{\xi}) &=: R^{\xi} \\
            \bar{c}^{\xi} - \Delta \bar{c}^{\xi} &= \bar{\rho}^{\xi}
        \end{aligned}
        & \text{in} \ (0,T)\times\Omega.
    \end{matrix}
\end{equation}
\begin{remark} [Stability framework]
    The problem \eqref{eq:strong_formulation} falls into the situation we established in Section \ref{section:a_stability_framework}. Thus, the stability framework of Theorem \ref{thm:stabilityest} is applicable to control the difference $\rho - \bar{\rho}^{\xi}$ in terms of the initial data and the residual $R^{\xi}$.
\end{remark}
Since the reconstructions are obtained as exact solutions to elliptic problems, estimating the norm of the residual requires some work.

\section{Controlling the $H^{-1}$-norm of the residual }\label{section:aposteriori}
The aim of this section is to establish a computable upper bound for $\left\| \mathcal{R^{\xi}}(t, \cdot) \right\|_{H^{-1}(\Omega)}$ for a.e. $t \in [0,T]$.
Let $\left\{ \mathcal{T}_h \right\}_{h>0}$ be an admissible mesh sequence.
We first recall the $L^2$-orthogonal projection and its approximation properties:
\begin{definition}[$L^2$-orthogonal projection] 
    Let $\pi_h$ denote the \emph{$L^2$-orthogonal projection} onto $V_h$, that is, $\pi_h: L^2(\Omega) \rightarrow$ $V_h$ is defined such that, for all $v \in L^2(\Omega)$, $\pi_h v $ belongs to $V_h$ and satisfies
\begin{equation*}
\left(\pi_h v, y_h\right)%_{L^2(\Omega)}
=\left(v, y_h\right)%_{L^2(\Omega)} 
\quad \forall y_h \in V_h .
\end{equation*}
\end{definition}

\begin{lemma} \label{thm:opL2proj} \cite[Lemma 1.58]{Di_Pietro_2012}
    Let $\pi_h$ be the $L^2$-orthogonal projection onto $V_h$, and let $s \in\{0, \ldots, k+1\}$. Then, for all $T \in \mathcal{T}_h$, and all $v \in H^s(T)$, we have
    \begin{equation} \nonumber
        \left|v-\pi_h v\right|_{H^m(T)} \leq C_{\text {app }}^{\prime} h_T^{s-m}|v|_{H^s(T)} \quad \forall m \in\{0, \ldots, s\}
    \end{equation}
    where the positive real number $C_{\text {app }}^{\prime}$ is independent of both $T$ and $h$.
\end{lemma}

\begin{lemma} \label{thm:opL2projmf} \cite[Lemma 1.59]{Di_Pietro_2012}
     Under the hypotheses of Lemma \ref{thm:opL2proj}, assume additionally that $s \geq 1$. Then, for all $T \in \mathcal{T}_h$, and all $F \in \mathcal{F}_T$, we have
\begin{equation*}
\left\|v-\pi_h v\right\|_{L^2(F)} \leq C_{\mathrm{app}}^{\prime \prime} h_T^{s-1 / 2}|v|_{H^s(T)},
\end{equation*}
and if $s \geq 2$,
\begin{equation*}
\left\|\left.\nabla\left(v-\pi_h v\right)\right|_T \cdot \mathrm{n}_T\right\|_{L^2(F)} \leq C_{\mathrm{app}}^{\prime \prime \prime} h_T^{s-3 / 2}|v|_{H^s(T)},
\end{equation*}
where the positive numbers $C_{\mathrm{app}}^{\prime \prime}$ and $C_{\mathrm{app}}^{\prime \prime \prime}$ are independent of both $T$ and $h$.
\end{lemma}

We now provide a computable upper bound for $\left\| \mathcal{R^{\xi}}(t, \cdot) \right\|_{H^{-1}(\Omega)}$ for a.e. $t \in [0,T]$. For brevity, we omit the time dependency:
\begin{lemma} [a posteriori control on $R^{\xi}$]\label{lem:Erho} 
    Let $\left( \bar{\rho}^{\tau}, \bar{c}^{\tau} \right)$ be the temporal reconstruction of the numerical solution $\left\{ (\rho_h^n, c_h^n) \right\}_{n=0}^N$ obtained from \eqref{eq:fully_discrete}. Let $R^{\tau}$ be the temporal residual as defined in \eqref{eq:temporal_residual}
    and $R^{\xi}$ the space-time residual as specified in \eqref{eq:residualpair}. Then, for $k\geq1$, we have the following estimate:
    \begin{equation} \label{eq:R_rho}
        \begin{aligned}
            &\left\|R^{\xi}\right\|_{H^{-1}(\Omega)} \\
            &\leq C_{*}\operatorname{E}_*\left[\partial_t \bar{\rho}^{\tau}%, \partial_t\left(A_h \bar{\rho}^{\tau}\right)
            \right] \\
            &\quad + 2C_S'' C_{ell} C_0 \operatorname{E}_0\left[\bar{\rho}^{\tau}%, A_h \bar{\rho}^{\tau}
            \right]
            \bigg( C_1^2 \operatorname{E}_1\left[\bar{\rho}^{\tau}%, A_h \bar{\rho}^{\tau}
            \right]^2 + \sum_{T \in \mathcal{T}_h} \left| \bar{\rho}^{\tau} \right|_{H^1(T)}^2 \bigg)^{1/2} \\
            &\quad + \left( \sum_{T \in \mathcal{T}_h} C_{app}'^2 h_T^2 \left\| \nabla_h \cdot \left( \bar{\rho}^{\tau}\nabla_h \bar{c}^{\tau} \right) - \pi_h\bigg( \nabla_h \cdot\left(\bar{\rho}^{\tau} \nabla_h \bar{c}^{\tau}\right) \bigg) \right\|_{L^2(T)}^2 \right)^{1/2}\\
            &\quad + \sqrt{2}N_\partial^{1/2}\left\{ \left( C_{app}' + 1 \right)^2 + 1 \right\}^{1/2} C_{max}^{1/2} \left\| \bar{\rho}^{\tau} \right\|_{L^\infty(\Omega)} \left( C_{ell}^2C_0^2\operatorname{E}_0[\bar{\rho}^{\tau}%, A_h\bar{\rho}^{\tau}
            ]^2 + \tilde{C}_1^2\operatorname{\tilde{E}}_1[\bar{c}^{\tau}, \bar{\rho}^{\tau}]^2 \right)^{1/2} \\
            &\quad + N_\partial C_{app}'' \left\| \nabla_h \bar{c}^{\tau} \right\|_{L^\infty(\Omega)} \bigg( \sum_{F \in \mathcal{F}_h^i} h_T \left\| \llbracket \bar{\rho}^{\tau} \rrbracket \right\|_{L^2(F)}^2 \bigg)^{1/2} + \left\| R^{\tau} \right\|_{L^2(\Omega)}\\
            &=: \operatorname{E}_{R^{\xi}}
        \end{aligned}
    \end{equation}
    where 
    the subscript $*$ of $C_*$ and $\operatorname{E}_*$ is defined as $0$ if $k=1$ and as $-1$ if $k\geq 2$.
    % \begin{equation} \nonumber
    %     \operatorname{E}_*:=
    %     \begin{cases}
    %         \operatorname{E}_{-1} & \text{if } k \geq 2, \\
    %         \operatorname{E}_0 & \text{if } k = 1,
    %     \end{cases}
    % \end{equation}
    % and $C_{*}$ is $C_0$ if $k=1$ and $C_{-1}$ if $k\geq 2$.
    The estimators $\operatorname{E}_{-1}, \operatorname{E}_0$, and $\operatorname{\tilde{E}}_1$ are defined in equations \eqref{rhotest2}, \eqref{rhoest2}, and \eqref{cest}, respectively. The constant $N_{\partial}$ is defined as $N_{\partial} := \max_{T \in \mathcal{T}_h} \operatorname{card}(\mathcal{F}_T)$. Constants $C_{app}'$ and $C_{app}''$ are specified in Lemmas \ref{thm:opL2proj} and \ref{thm:opL2projmf}, respectively. $C_S''$ is the embedding constant from $H^2$ to $L^\infty$, and $C_{max}$ is defined within the proof.
\end{lemma}

Prior to proving the lemma, there are several remarks to be made.

\begin{remark} [$L^2$-orthogonality]
    Since $\nabla_h \cdot (\bar{\rho}^{\tau} \nabla_h \bar{c}^{\tau})$ is a polynomial of degree $2k-2$, the third term in the right-hand side of \eqref{eq:R_rho} vanishes for $k \leq 2$, i.e.
    \begin{equation*}
        \sum_{T \in \mathcal{T}_h} C_{app}'^2 h_T^2 \left\| \nabla_h \cdot \left( \bar{\rho}^{\tau}\nabla_h \bar{c}^{\tau} \right) - \pi_h\left( \nabla_h \cdot\left(\bar{\rho}^{\tau} \nabla_h \bar{c}^{\tau}\right) \right) \right\|_{L^2(T)}^2=0.
    \end{equation*}
    
\end{remark}

\begin{remark}[Optimality of the estimator] \label{rmk:optimal}
    The term $\left\| \operatorname{E}_{R^{\xi}} \right\|_{L^2(0,T)}$ enters linearly into the error estimator \eqref{eq:fullestimator} for the error of our dG scheme measured in the norm $\sqrt{\left\| \cdot \right\|_{L^\infty(L^2)}^2 + \int_0^T \left\| \cdot \right\|_{dG}^2dt}$, and we have used first-order time discretisation. Thus, for a sufficiently smooth exact solution, the optimal scaling of $\operatorname{E}_{R^{\xi}}$ would be $h^k+\tau^1$.
    Indeed, the norms $\left\|\bar{\rho}^{\tau}\right\|_{L^{\infty}(\Omega)}$ and $\left\|\nabla_h \bar{c}^{\tau}\right\|_{L^{\infty}(\Omega)}$ are bounded uniformly in $h$ before blow-up in our simulations. Moreover, the solutions of the elliptic problems \eqref{eq:reconstruct}, \eqref{eq:reconstruct2} and \eqref{eq:reconstruct3} are sufficiently regular so that $\operatorname{E}_0$ is of order $h^{k+1}$, and $\operatorname{E}_1$ and $\tilde{\operatorname{E}}_1$ are of order $h^k$. When $\operatorname{E}_{0}$ and $\operatorname{E}_{-1}$ are evaluated on $\partial_t \bar{\rho}^{\tau}$, their orders are $h^{k+1}$ and $h^{k+2}$, respectively; See Tables \ref{tab:time_derivative_Es} and \ref{tab:time_derivative_Es_2}. Additionally,
    \begin{equation}\nonumber
        \left(\sum_{F \in \mathcal{F}_h^i} h_F \left\| \llbracket \bar{\rho}^{\tau} \rrbracket \right\|_{L^2(F)}^2\right)^{1/2}
    \end{equation}
    is of order $h^{k+1}$. The term $\left\| R^{\tau} \right\|_{L^2(\Omega)}$ is expected to be of order $\tau^1$.
\end{remark}

\begin{proof}[Proof of Lemma \ref{lem:Erho}]
    
Let $\pi_h$ be the $L^2$-orthogonal projection onto $V_h$. For $\phi \in H^1(\Omega)$,
we have
\begin{equation} \label{eq:residual}
    \begin{aligned}
        \left\langle R^{\xi}, \phi\right\rangle%_{H^{-1}, H^1}
        =&\left\langle \partial_t\bar{\rho}^{\xi}, \phi\right\rangle%_{{H^{-1}, H^1}}
        - (\bar{\rho}^{\xi} \nabla \bar{c}^{\xi}, \nabla\phi)%_{L^2(\Omega)}
        + (\nabla \bar{\rho}^{\xi}, \nabla\phi)%_{L^2(\Omega)}
        \\
        =& \left\langle \partial_t\bar{\rho}^{\xi}, \phi\right\rangle%_{{H^{-1}, H^1}} 
        - (\bar{\rho}^{\xi} \nabla \bar{c}^{\xi}, \nabla\phi)%_{L^2(\Omega)} 
        + \left( A_h\bar{\rho}^{\tau}, \phi \right)%_{L^2(\Omega)},
        \\
        %=& \left\langle \partial_t\bar{\rho}^{\xi}, \phi\right\rangle_{{H^{-1}, H^1}} - (\bar{\rho}^{\xi} \nabla \bar{c}^{\xi}, \nabla\phi)_{L^2(\Omega)} + \left( A_h\bar{\rho}^{\tau}, \pi_h\phi \right)_{L^2(\Omega)} \\
        =& \left\langle \partial_t\bar{\rho}^{\xi}, \phi\right\rangle%_{{H^{-1}, H^1}} 
        - (\bar{\rho}^{\xi} \nabla \bar{c}^{\xi}, \nabla\phi)%_{L^2(\Omega)} 
        + a^{sip}_h(\bar{\rho}^{\tau}, \pi_h\phi) \\
        =& \left\langle \partial_t\bar{\rho}^{\xi}, \phi\right\rangle%_{{H^{-1}, H^1}} 
        - (\bar{\rho}^{\xi} \nabla \bar{c}^{\xi}, \nabla\phi)%_{L^2(\Omega)} 
        + a^{wsip}_h(\bar{\rho}^{\tau}; \bar{c}^{\tau}, \pi_h\phi) - (\partial_t \bar{\rho}^{\tau}, \pi_h\phi) + \left( R^{\tau}, \pi_h \phi \right)
        \\
        %\text{since $(\rho_{h,t}, \phi - \pi_h\phi)=0$}&\ %\text{owing to $L^2$-orthogonality},  \\
    \end{aligned}
\end{equation}
where we utilized $L^2$-orthogonality in the third equality and equation \eqref{eq:temporal_residual} in the last equality.
By applying $L^2$-orthogonality once more, we obtain:
\begin{equation} \nonumber
    \begin{aligned}
        \left( \partial_t \bar{\rho}^{\tau}, \pi_h \phi \right)%_{L^2(\Omega)}
        &= \left( \partial_t \bar{\rho}^{\tau}, \phi \right)%_{L^2(\Omega)} 
        = \left\langle \partial_t \bar{\rho}^{\tau}, \phi \right\rangle, \\
        (R^{\tau}, \pi_h \phi) &= (R^{\tau}, \phi).        
    \end{aligned}
\end{equation}
Hence we can rewrite \eqref{eq:residual} as follows:
    \begin{equation} \nonumber
         \left\langle R^{\xi}, \phi\right\rangle%_{H^{-1}, H^1}
         = \underbrace{\left\langle \partial_t\bar{\rho}^{\xi} - \partial_t\bar{\rho}^{\tau}, \phi\right\rangle}_{:=I_1}%_{H^{-1}, H^1} 
         \\
         \underbrace{- (\bar{\rho}^{\xi} \nabla \bar{c}^{\xi}, \nabla\phi)
         + a^{wsip}_h(\bar{\rho}^{\tau}; \bar{c}^{\tau}, \pi_h\phi)}_{:=I_2} + \left( R^{\tau}, \phi \right).
    \end{equation}
It is obvious that
\begin{equation} \nonumber
    \left( R^{\tau}, \phi \right) \leq \left\| R^{\tau} \right\|_{L^2(\Omega)}\left\| \phi \right\|_{L^2(\Omega)} \leq \left\| R^{\tau} \right\|_{L^2(\Omega)}\left\| \phi \right\|_{H^1(\Omega)} .
\end{equation}
We provide separate bounds for $I_1$ and $I_2$ as follows:
Firstly, using Cauchy-Schwarz inequality and \eqref{rhotest2}, we estimate $I_1$
    \begin{equation} \nonumber
        I_1 \leq 
        C_{-1}\operatorname{E}_{-1}\left[\partial_t \bar{\rho}^{\tau}, \partial_t\left(A_h \bar{\rho}^{\tau}\right)\right] \left\| \phi \right\|_{H^1(\Omega)}.
    \end{equation}
Secondly, proceeding from the definition \eqref{eq:wsip}, we write
\begin{equation} \label{eq:compute_residual}
    \begin{aligned}
        I_2 &= -\left( \bar{\rho}^{\xi}\nabla\bar{c}^{\xi} - \bar{\rho}^{\tau}\nabla_h \bar{c}^{\tau}, \nabla\phi \right)%_{L^2(\Omega)}
        \\
        &\quad - \sum_{T \in \mathcal{T}_h}\int_T \left( \bar{\rho}^{\tau} \nabla_h \bar{c}^{\tau} \right)\cdot\left( \nabla\phi - \nabla_h \pi_h\phi\right) + \sum_{F \in \mathcal{F}_h^i}\int_F \sigma \frac{\gamma_{\bar{\rho}^{\tau}}}{h_F}\llbracket \bar{c}^{\tau} \rrbracket \llbracket \pi_h\phi \rrbracket \\
        &\quad -\sum_{F \in \mathcal{F}_h^i}\int_F\llbracket \pi_h\phi \rrbracket\dc{ \bar{\rho}^{\tau} \nabla_h \bar{c}^{\tau}  }_{\omega}\cdot \mathrm{n}_F - \sum_{F \in \mathcal{F}_h^i}\int_F \llbracket \bar{c}^{\tau} \rrbracket \dc{ \bar{\rho}^{\tau} \nabla_h \pi_h\phi  }_{\omega}\cdot \mathrm{n}_F. 
    \end{aligned}
\end{equation}
Next, we consider the first two terms of $I_2$:
\begin{multline} \nonumber
    -\left( \bar{\rho}^{\xi}\nabla \bar{c}^{\xi} - \bar{\rho}^{\tau} \nabla_h \bar{c}^{\tau}, \nabla \phi \right)%_{L^2(\Omega)}
    \\
    = - \sum_{T \in \mathcal{T}_h} \int_T \left( \bar{\rho}^{\xi} - \bar{\rho}^{\tau} \right)\nabla \bar{c}^{\xi} \cdot \nabla \phi - \sum_{T \in \mathcal{T}_h} \int_T \bar{\rho}^{\tau} \left( \nabla \bar{c}^{\xi} - \nabla_h \bar{c}^{\tau} \right)\cdot \nabla \phi
\end{multline}
and
\begin{multline} \label{eq:compute_residual2}
        - \sum_{T \in \mathcal{T}_h}\int_T \left( \bar{\rho}^{\tau} \nabla_h \bar{c}^{\tau} \right)\cdot\left( \nabla\phi - \nabla_h \pi_h\phi\right) \\
        = \sum_{T \in \mathcal{T}_h} \int_T \nabla_h \cdot \left( \bar{\rho}^{\tau} \nabla_h \bar{c}^{\tau} \right)\left( \phi -\pi_h\phi \right) 
       - \sum_{F \in \mathcal{F}_h^i} \int_F \llbracket \bar{\rho}^{\tau} \nabla_h \bar{c}^{\tau}  \cdot \mathrm{n}_F \left( \phi -\pi_h\phi \right) \rrbracket.
\end{multline}
Additionally, utilizing the following identity
\begin{equation} \label{eq:jump_identity}
     a_1 b_1-a_2 b_2 \\
     =\left(\omega_1 a_1+\omega_2 a_2\right)\left(b_1-b_2\right)+\left(a_1-a_2\right)\left(\omega_2 b_1+\omega_1 b_2\right)
\end{equation}
where $\omega_1$ and $\omega_2$ are positive real numbers such that $\omega_1+\omega_2=1$, we can rewrite the second term on the right-hand side of \eqref{eq:compute_residual2} as:
\begin{multline*}
        - \sum_{F \in \mathcal{F}_h^i}  \int_F \llbracket \bar{\rho}^{\tau} \nabla_h \bar{c}^{\tau} \cdot \mathrm{n}_F \left( \phi - \pi_h\phi \right) \rrbracket =\\ 
         - \sum_{F \in \mathcal{F}_h^i} \int_F \dc{ \bar{\rho}^{\tau} \nabla_h \bar{c}^{\tau}  }_{\omega}\cdot \mathrm{n}_F \llbracket \phi - \pi_h \phi \rrbracket 
         - \sum_{F \in \mathcal{F}_h^i} \int_F \llbracket \bar{\rho}^{\tau} \nabla_n \bar{c}^{\tau} \rrbracket \cdot \mathrm{n}_F \dc{ \phi - \pi_h \phi   }_{\bar{\omega}}.
    \end{multline*}
Combining the above calculations, we can rewrite equation \eqref{eq:compute_residual} as:
\begin{equation} \label{eq:residual3}
    \begin{aligned}
        I_2 = & -\sum_{T \in \mathcal{T}_h} \int_T \left( \bar{\rho}^{\xi} - \bar{\rho}^{\tau} \right)\nabla \bar{c}^{\xi} \cdot \nabla \phi - \sum_{T \in \mathcal{T}_h} \int_T \bar{\rho}^{\tau} \left( \nabla \bar{c}^{\xi} - \nabla_h \bar{c}^{\tau} \right)\cdot \nabla \phi \\
        & + \sum_{T \in \mathcal{T}_h} \int_T \nabla_h \cdot \left( \bar{\rho}^{\tau} \nabla_h \bar{c}^{\tau} \right)\left( \phi -\pi_h\phi \right) \\
        & -\sum_{F \in \mathcal{F}_h^i} \int_F \llbracket \bar{\rho}^{\tau} \nabla_n \bar{c}^{\tau} \rrbracket \cdot \mathrm{n}_F \dc{ \phi - \pi_h \phi   }_{\bar{\omega}}. \\
        & + \sum_{F \in \mathcal{F}_h^i}\sigma \frac{\gamma_{\bar{\rho}^{\tau}}}{h_F}\int_F \llbracket \bar{c}^{\tau} \rrbracket \llbracket \pi_h\phi \rrbracket \\
        & - \sum_{F \in \mathcal{F}_h^i}\int_F \llbracket \bar{c}^{\tau} \rrbracket \dc{ \bar{\rho}^{\tau} \nabla_h \pi_h\phi  }_{\omega}\cdot \mathrm{n}_F \\
        =:& A_1+A_2+A_3+A_4+A_5+A_6.
    \end{aligned}
\end{equation}
Note that, in summing up \eqref{eq:residual3}, we have used the following identity:
\begin{equation} \nonumber
    \sum_{F \in \mathcal{F}_h^i}\int_F\llbracket \pi_h\phi \rrbracket\dc{ \bar{\rho}^{\tau} \nabla_h \bar{c}^{\tau}  }_{\omega}\cdot \mathrm{n}_F = \sum_{F \in \mathcal{F}_h}\int_F \llbracket \pi_h \phi - \phi \rrbracket \dc{\bar{\rho}^{\tau} \nabla_h\bar{c}^{\tau}}_{\omega}\cdot \mathrm{n}_F
\end{equation}
for all $\phi \in H^1(\Omega)$.

From now on, we will estimate each $A_i$ separately.

\paragraph{Estimate for $A_1$:}
We have
    \begin{equation} \nonumber
        \begin{aligned}
            \int_\Omega \left( \bar{\rho}^{\xi} - \bar{\rho}^{\tau} \right)\nabla \bar{c}^{\xi} \cdot \nabla \phi \leq \left\|\bar{\rho}^{\xi} - \bar{\rho}^{\tau}\right\|_{L^2(\Omega)} \left\| \nabla \bar{c}^{\xi} \right\|_{L^\infty(\Omega)} |\phi|_{H^1(\Omega)}.
        \end{aligned}
    \end{equation}
    Since $d \leq 3$, we have
    \begin{equation} \label{ineq:gradcxi}
        \left\| \nabla \bar{c}^{\xi} \right\|_{L^\infty(\Omega)} \leq C_S'' \left\| \nabla \bar{c}^{\xi} \right\|_{H^2(\Omega)}
    \end{equation}
    where $C_S''$ is the constant of the embedding $H^2 \rightarrow L^\infty$. 
    By employing elliptic regularity, we obtain
    \begin{equation} \label{ineq:gradcxi2}
            \left\| \nabla \bar{c}^{\xi} \right\|_{H^2(\Omega)} \leq C_{ell} \left\| \nabla \bar{\rho}^{\xi} \right\|_{L^2(\Omega)} \\
            \leq 2 C_{ell}\left( \sum_{T \in \mathcal{T}_h} \left| \bar{\rho}^{\xi} - \bar{\rho}^{\tau} \right|_{H^1(T)}^2 + \left| \bar{\rho}^{\tau} \right|_{H^1(T)}^2\right)^{\tfrac12}
    \end{equation}
    where $C_{ell}$ is the constant of elliptic regularity.
    Now combining \eqref{rhoest1}, \eqref{rhoxest1}, \eqref{ineq:gradcxi}, and \eqref{ineq:gradcxi2}, we get
    \begin{equation} \nonumber
        \begin{aligned}
            &\int_\Omega \left( \bar{\rho}^{\xi} - \bar{\rho}^{\tau} \right)\nabla \bar{c}^{\xi} \cdot \nabla \phi \\
            % &\leq 2 C_S'' C_{ell}\left\|\bar{\rho}^{\xi} - \bar{\rho}^{\tau}\right\|_{L^2(\Omega)} \left( \sum_{T \in \mathcal{T}_h} \left| \bar{\rho}^{\xi} - \bar{\rho}^{\tau} \right|_{H^1(T)}^2 + \left| \bar{\rho}^{\tau} \right|_{H^1(T)}^2\right)^{1/2} |\phi|_{H^1(\Omega)} \\
            &\leq 2C_S'' C_{ell} C_0 \operatorname{E}_0\left[\bar{\rho}^{\tau}%, A_h \bar{\rho}^{\tau}
            \right] \left( C_1^2 \operatorname{E}_1\left[\bar{\rho}^{\tau}%, A_h \bar{\rho}^{\tau}
            \right]^2 + \sum_{T \in \mathcal{T}_h} \left| \bar{\rho}^{\tau} \right|_{H^1(T)}^2 \right)^{1/2} |\phi|_{H^1(\Omega)}.
    \end{aligned}
    \end{equation}

 \paragraph{Estimate for $A_2$:}
    
    We have
    \begin{equation} \nonumber
            \sum_{T \in \mathcal{T}_h} \int_T \bar{\rho}^{\tau} \left( \nabla \bar{c}^{\xi} - \nabla_h \bar{c}^{\tau} \right)\cdot \nabla \phi 
            \leq \left\| \bar{\rho}^{\tau} \right\|_{L^\infty(\Omega)} \sum_{T \in \mathcal{T}_h} \left| \bar{c}^{\xi} - \bar{c}^{\tau} \right|_{H^1(T)} \left| \phi \right|_{H^1(T)} 
            =: A_2'.
    \end{equation}

 \paragraph{Estimate for $A_3$:}
    Using $L^2$-orthogonality $\left( \pi_h\left( \nabla_h \cdot\left(\bar{\rho}^{\tau} \nabla_h \bar{c}^{\tau}\right) \right), \phi-\pi_h \phi \right)_{L^2} = 0$, we obtain
    \begin{equation*}
        \begin{aligned}
            \sum_{T \in \mathcal{T}_h}& \int_T \nabla_h \cdot\left(\bar{\rho}^{\tau} \nabla_h \bar{c}^{\tau}\right)\left(\phi-\pi_h \phi\right) \\
            \leq&
            \sum_{T \in \mathcal{T}_h} \left\| \nabla_h \cdot \left( \bar{\rho}^{\tau}\nabla_h \bar{c}^{\tau} \right) - \pi_h\left( \nabla_h \cdot\left(\bar{\rho}^{\tau} \nabla_h \bar{c}^{\tau}\right) \right) \right\|_{L^2(T)} \left\| \phi - \pi_h \phi \right\|_{L^2(T)} .\\
            \text{By Le}&\text{mma \ref{thm:opL2proj}, we have} \\
            \leq & \sum_{T \in \mathcal{T}_h} C_{app}' h_T \left\| \nabla_h \cdot \left( \bar{\rho}^{\tau}\nabla_h \bar{c}^{\tau} \right)
            -\pi_h\left( \nabla_h \cdot\left(\bar{\rho}^{\tau} \nabla_h \bar{c}^{\tau}\right) \right)\right\|_{L^2(T)} \left| \phi \right|_{H^1(T)}. \\
            \text{Now a}&\text{pplying Cauchy-Schwarz inequality, we get}\\
            \leq & \left( \sum_{T \in \mathcal{T}_h} C_{app}'^2 h_T^2 \left\| \nabla_h \cdot \left( \bar{\rho}^{\tau}\nabla_h \bar{c}^{\tau} \right) - \pi_h\left( \nabla_h \cdot\left(\bar{\rho}^{\tau} \nabla_h \bar{c}^{\tau}\right) \right) \right\|_{L^2(T)}^2 \right)^{1/2}\left| \phi \right|_{H^1(\Omega)}.
        \end{aligned}
    \end{equation*}

\paragraph*{Estimate for $A_4$:}
We first observe that
    utilizing the identity \eqref{eq:jump_identity} with the choice of 
    \begin{equation*}
        \omega_1 = \frac{\bar{\rho}^{\tau-}}{\bar{\rho}^{\tau+} + \bar{\rho}^{\tau-}}, \omega_2 = \frac{\bar{\rho}^{\tau+}}{\bar{\rho}^{\tau+} + \bar{\rho}^{\tau-}}
    \end{equation*}
    yields
    \begin{equation} \nonumber
        \llbracket \bar{\rho}^{\tau} \nabla_h \bar{c}^{\tau} \rrbracket = \gamma_{\bar{\rho}^{\tau}} \llbracket \nabla_h \bar{c}^{\tau} \rrbracket + \llbracket \bar{\rho}^{\tau} \rrbracket \dc{ \nabla_h \bar{c}^{\tau}  }_{\bar{\omega}}.
    \end{equation}
    Moreover, we have
    $
        \gamma_{\bar{\rho}^{\tau}} \leq \left\| \bar{\rho}^{\tau} \right\|_{L^\infty(\Omega)}
    $
    for all $F \in \mathcal{F}_h^i$.
    Thus, we obtain
    \begin{equation*}
        \begin{aligned}
            &\sum_{F \in \mathcal{F}^i_h} \int_F \llbracket \bar{\rho}^{\tau}\nabla_h \bar{c}^{\tau} \rrbracket\cdot \mathrm{n}_F \dc{ \phi - \pi_h\phi  }_{\bar{\omega}} \\
            % &= \sum_{F \in \mathcal{F}^i_h} \gamma_{\bar{\rho}^{\tau}} \llbracket \nabla_h \bar{c}^{\tau} \rrbracket \cdot \mathrm{n}_F \dc{ \phi - \pi_h\phi  }_{\bar{\omega}} + \sum_{F \in \mathcal{F}^i_h} \llbracket \bar{\rho}^{\tau} \rrbracket \dc{ \nabla_h \bar{c}^{\tau}  }_{\bar{\omega}} \cdot \mathrm{n}_F \dc{ \phi - \pi_h\phi  }_{\bar{\omega}} \\
            &= \sum_{T \in \mathcal{T}_h}\sum_{F \in \mathcal{F}_T} \int_F \bigl( \gamma_{\bar{\rho}^{\tau}}\llbracket \nabla_h \bar{c}^{\tau} \rrbracket + \llbracket \bar{\rho}^{\tau} \rrbracket \dc{ \nabla_h \bar{c}^{\tau}  }_{\bar{\omega}} \bigr) \cdot \mathrm{n}_F \bar{\omega}_{T,F}\left( \phi -\pi_h \phi \right)\\
            % &\quad + \sum_{T \in \mathcal{T}_h}\sum_{F \in \mathcal{F}_T} \int_F \llbracket \bar{\rho}^{\tau} \rrbracket \dc{ \nabla_h \bar{c}^{\tau}  }_{\bar{\omega}} \cdot \mathrm{n}_F \bar{\omega}_{T,F}\left( \phi -\pi_h \phi \right) \\
            &\leq \left\| \bar{\rho}^{\tau} \right\|_{L^\infty(\Omega)} \sum_{T \in \mathcal{T}_h}\sum_{F \in \mathcal{F}_T}  \left\| \llbracket \nabla_h \bar{c}^{\tau} \rrbracket \cdot \mathrm{n}_F \right\|_{L^2(F)} \left\| \phi - \pi_h \phi \right\|_{L^2(F)} \\
            &\quad + \left\| \nabla_h \bar{c}^{\tau} \right\|_{L^\infty(\Omega)} \sum_{T \in \mathcal{T}_h}\sum_{F \in \mathcal{F}_T} \left\| \llbracket \bar{\rho}^{\tau} \rrbracket \right\|_{L^2(F)} \left\| \phi - \pi_h \phi \right\|_{L^2(F)}. 
            \end{aligned}
    \end{equation*}  
        Now applying Lemma \ref{thm:opL2projmf} gives
    \begin{multline*}
            \sum_{F \in \mathcal{F}^i_h} \int_F \llbracket \bar{\rho}^{\tau}\nabla_h \bar{c}^{\tau} \rrbracket\cdot \mathrm{n}_F \dc{ \phi - \pi_h\phi  }_{\bar{\omega}} \\
            \leq \left\| \bar{\rho}^{\tau} \right\|_{L^\infty(\Omega)} \sum_{T \in \mathcal{T}_h}\sum_{F \in \mathcal{F}_T} C_{app}'' h_T^{1/2} \left\| \llbracket \nabla_h \bar{c}^{\tau} \rrbracket \cdot \mathrm{n}_F \right\|_{L^2(F)} | \phi |_{H^1(T)} \\
             + \left\| \nabla_h \bar{c}^{\tau} \right\|_{L^\infty(\Omega)} \sum_{T \in \mathcal{T}_h}\sum_{F \in \mathcal{F}_T} C_{app}'' h_T^{1/2} \left\| \llbracket \bar{\rho}^{\tau} \rrbracket \right\|_{L^2(F)} | \phi |_{H^1(T)} 
            =: A_4',
        \end{multline*}
    where $\bar{\omega}_{T,F}$ is the corresponding weight coefficient on each $T\in \mathcal{T}_h$ and $F \in \mathcal{F}_T$.
    
 \paragraph{Estimate for $A_5$:}
    We have
    % Using $\epsilon_{T,F}:= \pm 1$ depending on $T$ and $F$, we estimate
    \begin{equation} \nonumber
        \begin{aligned}
            &\sum_{F \in \mathcal{F}_h^i} \sigma\frac{\gamma_{\bar{\rho}^{\tau}}}{h_F}\int_F \llbracket \bar{c}^{\tau} \rrbracket \llbracket \phi - \pi_h \phi \rrbracket
            \\
            &= \sum_{T \in \mathcal{T}_h} \sum_{F \in \mathcal{F}_T} \sigma\frac{\gamma_{\bar{\rho}^{\tau}}}{h_F} \epsilon_{T, F}\int_{F} \llbracket \bar{c}^{\tau} \rrbracket \left( \phi - \pi_h \phi|_T \right) \\
            &\leq \left\| \bar{\rho}^{\tau} \right\|_{L^\infty(\Omega)}\sum_{T \in \mathcal{T}_h} \sum_{F \in \mathcal{F}_T} \frac{\sigma}{h_F} \left\| \llbracket \bar{c}^{\tau} \rrbracket \right\|_{L^2(F)} \left\| \phi - \pi_h \phi|_T \right\|_{L^2(F)} \\
            \text{wh}& \text{ere $\epsilon_{T,F}$ is the number $1$ or $-1$ determined by the choice of $T \in \mathcal{T}_h$ and $F \in \mathcal{F}_T$.}\\ 
            \text{No}& \text{w using Lemma \ref{thm:opL2projmf} we obtain} \\
            &\leq \left\| \bar{\rho}^{\tau} \right\|_{L^\infty(\Omega)}\sum_{T \in \mathcal{T}_h} \sum_{F \in \mathcal{F}_T} C_{app}'' h_T^{1/2}\frac{\sigma}{h_F}\left\| \llbracket \bar{c}^{\tau} \rrbracket \right\|_{L^2(F)}\left| \phi \right|_{H^1(T)} 
            =:A_5'.
        \end{aligned} 
    \end{equation}

 \paragraph{Estimate for  $A_6$:}
    We have
    \begin{equation} \nonumber
        \begin{aligned}
            &\sum_{F \in \mathcal{F}_h^i} \int_F\llbracket \bar{c}^{\tau} \rrbracket \dc{\bar{\rho}^{\tau}\nabla_h \pi_h \phi }_{\omega} \cdot \mathrm{n}_F  \\
            &= \sum_{F \in \mathcal{F}_h^i} \int_F \gamma_{\bar{\rho}^{\tau}} \llbracket \bar{c}^{\tau} \rrbracket  \dc{ \nabla_h\pi_h\phi   } \cdot \mathrm{n}_F \\
            & = \frac{1}{2} \sum_{T \in \mathcal{T}_h}\sum_{F \in \mathcal{F}_T} \int_F \gamma_{\bar{\rho}^{\tau}} \llbracket \bar{c}^{\tau} \rrbracket \nabla_h \pi_h \phi \cdot \mathrm{n}_F \\
            &\leq \left\| \bar{\rho}^{\tau} \right\|_{L^\infty(\Omega)}\sum_{T \in \mathcal{T}_h}\sum_{F \in \mathcal{F}_T}  \left\| \llbracket \bar{c}^{\tau} \rrbracket \right\|_{L^2(F)} \left\| \nabla_h \pi_h \phi \cdot \mathrm{n}_F\right\|_{L^2(F)}  \\
            &\leq\left\| \bar{\rho}^{\tau} \right\|_{L^\infty(\Omega)}\sum_{T \in \mathcal{T}_h}\sum_{F \in \mathcal{F}_T} \frac{C_{tr}}{h_T^{1/2}} \left\| \llbracket \bar{c}^{\tau} \rrbracket \right\|_{L^2(F)} \left\| \nabla_h \pi_h \phi \right\|_{L^2(T)}  =:A_6'
        \end{aligned}
    \end{equation}
    where $C_{tr}$ is the constant of the trace inequality applied for $F \subset \partial T$.

By summing up the terms $A_2', A_4', A_5',$ and $A_6'$, applying the Cauchy-Schwarz inequality, and gathering relevant terms, we obtain:
\begin{equation*}
    \begin{aligned}
      &  A_2' + A_4' + A_5' + A_6' \leq \\& \left\| \bar{\rho}^{\tau} \right\|_{L^\infty(\Omega)} \left( \sum_{T \in \mathcal{T}_h} \left| \bar{c}^{\xi} - \bar{c}^{\tau}  \right|_{H^1(T)}^2 + \sum_{T \in \mathcal{T}_h}\sum_{F \in \mathcal{F}_T} C_{app}''^2 h_T \left\| \llbracket \nabla_h \bar{c}^{\tau} \rrbracket \cdot \mathrm{n}_F \right\|_{L^2(F)}^2 \right.\\ 
        & \left. + \sum_{T \in \mathcal{T}_h} \sum_{F \in \mathcal{F}_T} C_{app}''^2 h_T \frac{\sigma^2}{h_F^2} \left\| \llbracket \bar{c}^{\tau} \rrbracket \right\|_{L^2(F)}^2 + \sum_{T \in \mathcal{T}_h}\sum_{F \in \mathcal{F}_T} \frac{C_{tr}^2}{h_T} \left\| \llbracket \bar{c}^{\tau} \rrbracket \right\|_{L^2(F)}^2 \right)^{1/2} \\
        &\quad \times\left( \sum_{T\in \mathcal{T}_h}\sum_{F\in \mathcal{F}_T} \left| \phi \right|_{H^1(T)}^2 + \sum_{T \in \mathcal{T}_h}\sum_{F \in \mathcal{F}_T} \left\| \nabla_h \pi_h \phi  \right\|_{L^2(T)}^2 \right)^{1/2} \\
        & + \left\| \nabla_h \bar{c}^{\tau} \right\|_{L^\infty(\Omega)} \left( \sum_{T \in \mathcal{T}_h}\sum_{F \in \mathcal{F}_T} C_{app}''^2 h_T \left\| \llbracket \bar{\rho}^{\tau} \rrbracket \right\|_{L^2(F)}^2 \right)^{1/2} \left( \sum_{T\in \mathcal{T}_h}\sum_{F\in \mathcal{F}_T} \left| \phi \right|_{H^1(T)}^2 \right)^{1/2}.
    \end{aligned}
\end{equation*}
Since $\left\{ \mathcal{T}_h \right\}_{h>0}$ is shape- and contact-regular, there exists a mesh regularity parameter $\delta > 0$ such that for any $T \in \mathcal{T}_h$ and any $F \in \mathcal{F}_T$,
$
    \delta h_T \leq h_F
$. Moreover, we have
\begin{equation} \nonumber
    \left\| \nabla_h \pi_h \phi  \right\|_{L^2(T)} \leq \left( C_{app}' + 1 \right) \left| \phi \right|_{H^1(T)}.
\end{equation}
Now we can rewrite the expression as:
\begin{equation} \nonumber
    \begin{aligned}
        &A_2' + A_4' + A_5' + A_6' \leq \\&  
        \Bigg[ N_\partial^{1/2}\left\{ \left( C_{app}' + 1 \right)^2 + 1 \right\}^{1/2} \left\| \bar{\rho}^{\tau} \right\|_{L^\infty(\Omega)} \Bigg( \sum_{T \in \mathcal{T}_h} \left| \bar{c}^{\xi} - \bar{c}^{\tau}  \right|_{H^1(T)}^2 
        \\ 
        & + N_\partial C_{app}''^2\sum_{F \in \mathcal{F}_h^i}  h_F \left\| \llbracket \nabla_h \bar{c}^{\tau} \rrbracket \right\|_{L^2(F)}^2  + N_\partial \left( C_{app}''^2\sigma^2 \delta^{-1} + C_{tr}^2 \right) \sum_{F \in \mathcal{F}_h^i} h_F^{-1} \left\| \llbracket \bar{c}^{\tau} \rrbracket \right\|_{L^2(F)}^2 \Bigg)^{1/2}   
        \\
        & + N_\partial C_{app}'' \left\| \nabla_h \bar{c}^{\tau} \right\|_{L^\infty(\Omega)} \left( \sum_{F \in \mathcal{F}_h^i} h_F \left\| \llbracket \bar{\rho}^{\tau} \rrbracket \right\|_{L^2(F)}^2 \right)^{1/2}
        \Bigg]\left| \phi \right|_{H^1(\Omega)}
    \end{aligned}
\end{equation}
where $N_{\partial}:=\max _{T \in \mathcal{T}_h} \operatorname{card}\left(\mathcal{F}_T\right)$.
% Utilizing \eqref{eq:cest2}, we can further estimate the terms relevant to the error in $\bar{c}^{\tau}$:
Note that using \eqref{eq:cest2} we have:
\begin{equation} \nonumber
    \begin{aligned}
        & \sum_{T \in \mathcal{T}_h} \left| \bar{c}^{\xi} - \bar{c}^{\tau}  \right|_{H^1(T)}^2 + N_\partial C_{app}''\sum_{F \in \mathcal{F}_h^i}  h_F \left\| \llbracket \nabla_h \bar{c}^{\tau} \rrbracket \right\|_{L^2(F)}^2 \\ 
        &\quad  + N_\partial \left( C_{app}''^2\sigma^2 \delta^{-1} + C_{tr}^2 \right) \sum_{F \in \mathcal{F}_h^i} h_F^{-1} \left\| \llbracket \bar{c}^{\tau} \rrbracket \right\|_{L^2(F)}^2
        \\
        &\leq C_{max} \left( \sum_{T \in \mathcal{T}_h} \left| \bar{c}^{\xi} - \bar{c}^{\tau}  \right|_{H^1(T)}^2 + \sum_{F \in \mathcal{F}_h^i}  h_F \left\| \llbracket \nabla_h \bar{c}^{\tau} \rrbracket \right\|_{L^2(F)}^2 + \sigma^2\sum_{F \in \mathcal{F}_h^i} h_F^{-1} \left\| \llbracket \bar{c}^{\tau} \rrbracket \right\|_{L^2(F)}^2 \right) 
        \\
        &\leq C_{max} \left( C_{ell}^2C_0^2\operatorname{E}_0[\bar{\rho}^{\tau}%, A_h\bar{\rho}^{\tau}
        ]^2 + \tilde{C}_1^2\operatorname{\tilde{E}}_1[\bar{c}^{\tau}, \bar{\rho}^{\tau}]^2 \right)
    \end{aligned}
\end{equation}
where $C_{max}:= \max\left\{ 1, N_\partial C_{app}'', N_\partial \left( C_{app}''^2 \delta^{-1} + C_{tr}^2 \right) \right\}$.

In conclusion, we obtain the desired result by combining all terms.
% Thus, collecting all terms concludes the proof of the Theorem.
\end{proof}

\section{A posteriori error estimator}\label{sec:aposteriori}
We now present our main result, i.e., an a posteriori error estimator for the numerical approximation $\left\{ \left( \rho_h^n, c_h^n \right) \right\}_{n=0}^N$ obtained from the scheme \eqref{eq:fully_discrete}. The main idea is to split the norm of the error into two norms using the reconstruction and to estimate the norms using the stability framework (Section \ref{section:a_stability_framework} and \ref{section:aposteriori}) and the error estimators for reconstructions (Section \ref{sec:reconstruction}), respectively.

Since the stability estimate of Theorem \ref{thm:stabilityest} is conditional, it is necessary to establish computable upper bounds for quantities $A$ and $E$ of \eqref{eq:AE}. Thus we introduce the following quantities:
\begin{equation} \label{eq:barAbarE}
    \begin{aligned}
        A^{\xi}:=& 2 \left\| \rho_0 - \rho_h^0 \right\|_{L^2(\Omega)}^2 + 2C_0^2 \operatorname{E}_0[\rho_h^0]^2 + \int_0^T \operatorname{E}_{R^{\xi}}^2 (s) ds, \\
        E^{\xi}:=& \exp{\left( \int_0^T a^{\xi}(s) ds \right)},
    \end{aligned}
\end{equation}
where
\begin{multline*}
    a^{\xi}(s): = 6C_S^2C_S'^2C_{ell}^2\biggl( C_0^2 \operatorname{E}_0[\bar{\rho}^{\tau}(s)]^2 + C_1^2 \operatorname{E}_1[\bar{\rho}^{\tau}(s)]^2 
    + \left\| \bar{\rho}^{\tau}(s, \cdot) \right\|_{L^2(\Omega)}^2 + \left\| \bar{\rho}^{\tau}(s, \cdot) \right\|_{\mathrm{dG}} \biggr) \\
    + 6C_S''C_{ell}^2 \biggl( C_1^2 \operatorname{E}_1[\bar{\rho}^{\tau} (s)]^2 + \left\| \bar{\rho}^{\tau} (s) \right\|_{\mathrm{dG}} \biggr) + 1/3.
\end{multline*}
The estimators $\operatorname{E}_0$ and $\operatorname{E}_1$ are defined in \eqref{rhoest2} and \eqref{rhoxest1}, respectively. The estimator
$\operatorname{E}_{R^{\xi}}$ is defined as in \eqref{eq:R_rho}.
Constants $C_S$, $C_S'$, and $C_S''$ represent the embedding constants from $H^1 \rightarrow L^6$, $H^1 \rightarrow L^3$, and $H^2 \rightarrow L^\infty$, respectively. $C_{ell}$ denotes the constant of elliptic regularity, and $C_0$ and $C_1$ are the constants in \eqref{rhoest1} and \eqref{rhoxest1}, respectively.
It is straightforward that 
\begin{equation} \label{eq:AEbound}
    A \leq A^{\xi} \text{ and } E \leq E^{\xi}. 
\end{equation}

\begin{theorem}[A posteriori error estimate] \label{thm:aposteriori}
    Let $\left( \rho, c \right)$ be a weak solution  to \eqref{eq:KS1}--\eqref{eq:ic} and let $\left\{ (\rho_h^n, c_h^n) \right\}_{n=0}^N$ be the numerical solution from \eqref{eq:fully_discrete}. Let $A^{\xi}$ and $E^{\xi}$ be as defined in \eqref{eq:barAbarE}, and let $B = 2C_S'C_SC_{ell}$ as in Theorem \ref{thm:stabilityest}.
    If the condition 
    \begin{equation} \label{eq:condition}
        8 A^\xi E^\xi ( 8B (1 +T) E^\xi)^{2} \leq 1
    \end{equation}
    holds, we have
    \begin{multline}\label{eq:fullestimator}
            \left\| \rho - \rho_h^n \right\|_{L^\infty(0,T;L^2(\Omega))}^2 + \int_0^T \left\| \rho(s, \cdot) - \bar{\rho}^{\tau}(s, \cdot) \right\|_{\mathrm{dG}}^2 ds \\ 
            \leq 16 A^{\xi}E^{\xi}
            + 4 C_0^2 \left\| \operatorname{E}_0\left[ \bar{\rho}^{\tau}(s, \cdot) \right] \right\|_{L^\infty(0,T)}^2 + 4 \left\| \bar{\rho}^{\tau} - \rho_h^n \right\|_{L^\infty(0,T;L^2(\Omega))}^2 \\
            + 2C_1^2 \left\| \operatorname{E}_1\left[ \bar{\rho}^{\tau}(s, \cdot)
            \right] \right\|_{L^2(0,T)}^2
    \end{multline}
    where $\operatorname{E}_{0}$ and $\operatorname{E}_{1}$ are defined in \eqref{rhoest2}, \eqref{rhoxest1}, and \eqref{eq:R_rho}, respectively.
\end{theorem}
\begin{remark} [Optimality of the estimator]
    Based on the arguments given in Remark \ref{rmk:optimal}, the right-hand side of \eqref{eq:fullestimator} is expected to scale optimally, i.e., with the order $h^k + \tau$, as long as the exact solution is smooth.
    This is confirmed in our numerical experiments presented in Section \ref{sec:numerical}.
\end{remark}

\begin{remark} [Computable constants]
    All constants appearing in Theorem \ref{thm:aposteriori} can be evaluated on a convex polygonal domain. The Sobolev constants are provided in \cite{2017Mizuguchi}. For the constant of the Poincar\'e-type inequality, we refer to \cite{1960Payne}. The constant of elliptic regularity can also be explicitly computed \cite[Chapter 2]{1992Grisvard}. The constants for the discrete trace and inverse inequalities can be evaluated, see \cite{2003Warburton}. We refer to \cite[Theorem 5.45]{Di_Pietro_2012} for the constants of $E_1$. The constants of $E_{-1}$ and $E_0$ can be evaluated in a similar manner, see \cite{Makridakis2003}.
\end{remark}

\begin{proof}[Proof of Theorem \ref{thm:aposteriori}]
    Let $(\bar{\rho}^{\xi}, \bar{c}^{\xi})$ be the space-time reconstruction of $\left\{ (\rho_h^n, c_h^n) \right\}_{n=0}^N$.
    Then we have
    \begin{equation} \nonumber
        \begin{aligned}
            \left\| \rho - \rho_h^n \right\|_{L^\infty(0,T;L^2(\Omega))}^2 &\leq 2 \left\| \rho - \bar{\rho}^{\xi} \right\|_{L^\infty(0,T;L^2(\Omega))}^2 + 2 \left\| \bar{\rho}^{\xi} - \rho_h^n  \right\|_{L^\infty(0,T;L^2(\Omega))}^2 \\
            & =: I_1 + I_2
        \end{aligned}
    \end{equation}
    and
    \begin{multline} \nonumber
        \int_0^T \left\| \rho (t, \cdot) - \bar{\rho}^{\tau} (t, \cdot) \right\|_{\mathrm{dG}}^2 dt \\ 
        \leq 2 \int_0^T \left| \rho (t, \cdot) - \bar{\rho}^{\xi} (t, \cdot) \right|_{H^1(\Omega)}^2 dt + 2 \int_0^T \left\| \bar{\rho}^{\xi} (t, \cdot) - \bar{\rho}^{\tau} (t, \cdot) \right\|_{\mathrm{dG}}^2 dt \\
        =: J_1 + J_2.
    \end{multline}
    We first estimate the sum of the terms $I_1$ and $J_1$. Note that \eqref{eq:condition} guarantee that the condition \eqref{eq:estcondition} in Theorem \ref{thm:stabilityest} is satisfied.
    Thus we have
    \begin{equation} \nonumber
        \begin{aligned}
            I_1 + J_1 &\leq 16 \left(    \|  \rho (0,\cdot) - \bar{\rho}^{\xi} (0,\cdot)\|_{L^2(\Omega)}^2 + \int_0^{T} \| R^{\xi}(t,\cdot)\|_{H^{-1}(\Omega)}^2  dt   \right) \\
            &\times \exp \left( \int_0^T 3 C_S^2C_S'^2C_{ell}^2 \|\bar{\rho}^{\xi}(t,\cdot)\|_{H^1(\Omega)}^2  + 3 C_S''^2C_{ell}^2\|\nabla \bar{\rho}^{\xi}(t,\cdot)\|_{L^2(\Omega)}^2 + \frac 1 3 dt\right).
        \end{aligned}
    \end{equation}
    Now by \eqref{eq:AEbound}, we obtain
    \begin{equation} \nonumber
        I_1 + J_1 \leq 16 A^{\xi}E^{\xi}.
    \end{equation}
    Next, we estimate the term $I_2$ and $J_2$:
    \begin{equation} \nonumber
        \begin{aligned}
            I_2 &\leq 4 \left\| \bar{\rho}^{\xi} - \bar{\rho}^{\tau} \right\|_{L^\infty(0,T;L^2(\Omega))}^2 + 4 \left\| \bar{\rho}^{\tau}  - \rho_h^n \right\|_{L^\infty(0,T;L^2(\Omega))}^2 \\
            &\leq 4 C_0^2 \left\| \operatorname{E}_0\left[ \bar{\rho}^{\tau}(t, \cdot) \right] \right\|_{L^\infty(0,T)}^2  + 4 \left\| \bar{\rho}^{\tau}  - \rho_h^n \right\|_{L^\infty(0,T;L^2(\Omega))}^2
        \end{aligned}
    \end{equation}
    and
    \begin{equation} \nonumber
        J_2 \leq 2C_1^2 \left\| \operatorname{E}_1\left[ \bar{\rho}^{\tau}(t, \cdot) \right] \right\|_{L^2(0,T)}^2.
    \end{equation}

    Hence we obtain the desired result.
\end{proof}

\begin{remark}[Adaptivity]
    Adaptive strategies based on a posteriori error estimates containing analogous exponential terms can be found in
    \cite{2011Bartels, Cangiani_2016, 2023Giesselmann}.
    %, where adaptive algorithms are proposed in cases where  are presented.
\end{remark}

\section{Numerical experiments} \label{sec:numerical}
In this section, we show the results of numerical experiments to validate the optimal scaling behavior of the error estimator \eqref{eq:fullestimator}.
The implementations were carried out using Python.

\begin{definition} [Estimated order of convergence]
    Given two sequences $a(i)$ and $h(i) \searrow 0$, we define the \emph{estimated order of convergence (EOC)} to be the local slope of the $\log a(i)$ vs. $\log h(i)$ curve, i.e.,
    \begin{equation*}
    \operatorname{EOC}(a, h ; i):=\frac{\log (a(i+1) / a(i))}{\log (h(i+1) / h(i))}.
    \end{equation*}
\end{definition}

Let $\Omega \subset \mathbb{R}^2$ be the unit square $(0,1)\times(0,1)$ and let $\rho_0(x,y)$ be defined as
    \begin{equation} \nonumber
        \rho_0(x, y) = 10^{-3} \exp \left( - \frac{(x - x_0)^2 + (y - y_0)^2}{10^{-2}} \right).
    \end{equation}
We choose the center point $x_0 = 0.5$ and $y_0 = 0.5$.
The initial data $\rho_0$ yields a radially symmetric solution $\rho = \rho(t, x, y)$, which blows up in finite time at the center as illustrated in Figure \ref{fig:blowup}. A sequence of approximate solutions $\left\{ \rho_{h(i)}^n \right\}_{n=0}^{N(i)}$ obtained from \eqref{eq:fully_discrete} is computed using mesh width $h = 2^{-i}\operatorname{diam}(\Omega)$ and time step size $\tau = 2^{2-i}\operatorname{length}([0,T])$, where $i=4,5,\cdots, 9$. Here we fix $T=0.0045$.

\begin{figure} 
    \centering
    \begin{subfigure}{0.45\textwidth}
        \centering
        \includegraphics[width=\linewidth]{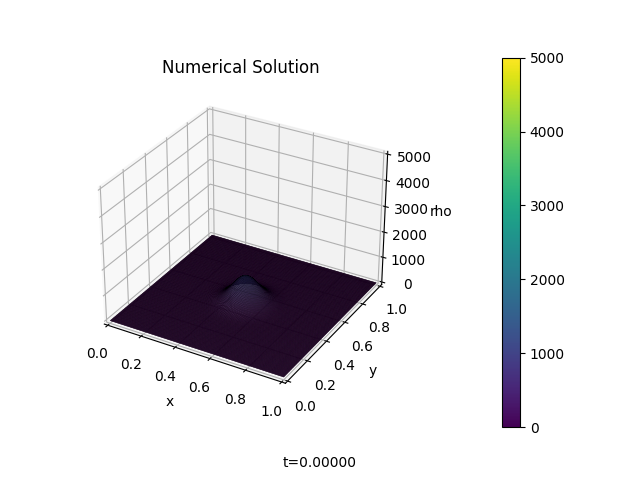}
        % \caption{$t = 0$}
        % \label{fig:subfig1}
    \end{subfigure}
    \hfill
    \begin{subfigure}{0.45\textwidth}
        \centering
        \includegraphics[width=\linewidth]{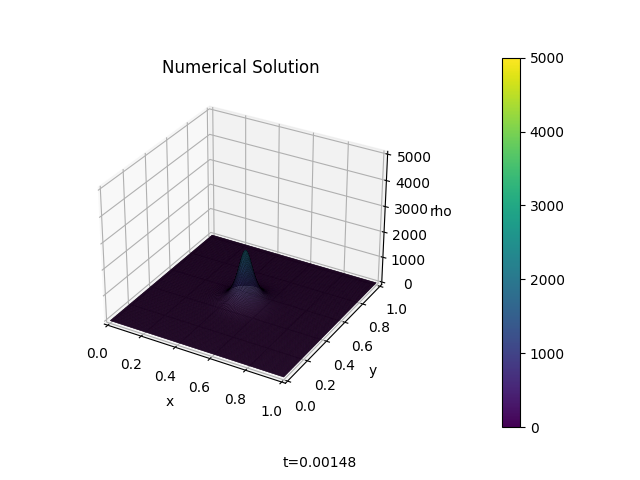}
        % \caption{$t = 0.0003$}
        % \label{fig:subfig2}
    \end{subfigure}

    \vspace{1em}
    \begin{subfigure}{0.45\textwidth}
        \centering
        \includegraphics[width=\linewidth]{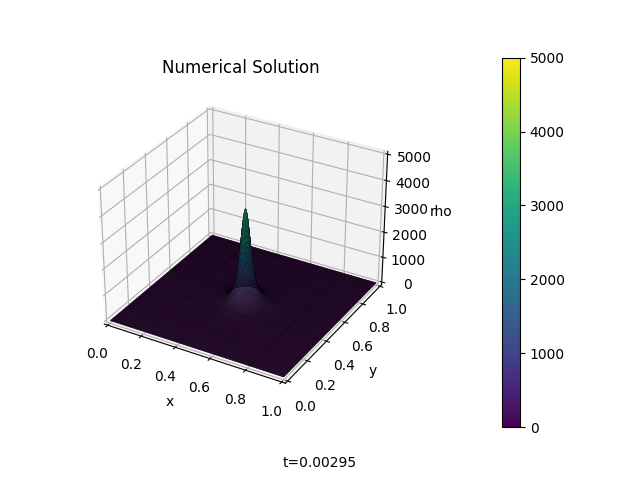}
        % \caption{$t = 0.0006$}
        % \label{fig:subfig3}
    \end{subfigure}
    \hfill
    \begin{subfigure}{0.45\textwidth}
        \centering
        \includegraphics[width=\linewidth]{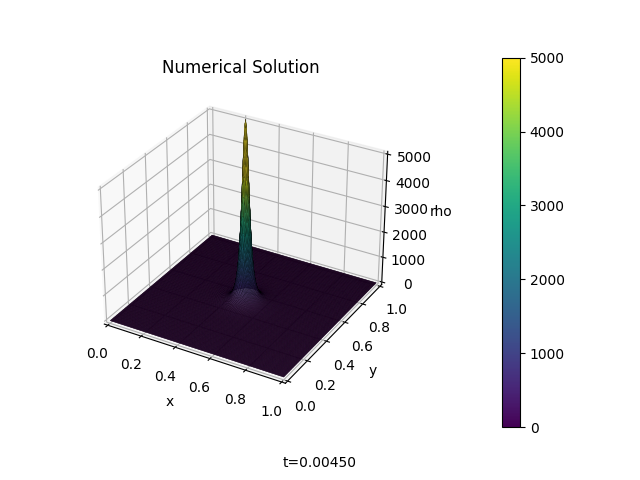}
        % \caption{$t = 0.0009$}
        % \label{fig:subfig4}
    \end{subfigure}

    \caption{Simulation snapshots for $\left\{ \rho_h^n \right\}_{n=0}^N$ with $k = 1$ and $i=8$ at different times.} 
    \label{fig:blowup}
    % \label{fig:main}
\end{figure}

Since the computation of \eqref{eq:fullestimator} results in overflows due to the presence of $E^{\xi}$, we only perform experiments on the key components: $\operatorname{E}_0$, $\operatorname{E}_1$, $\operatorname{E}_{-1}$, $\operatorname{E}_{R^{\xi}}$, and $\int_0^T a^{\xi}$. 
To compute the $L^2$-norms (in time) of the estimators, the reconstruction $(\bar{\rho}^{\tau}, \bar{c}^{\tau})$ of $\left\{ (\rho_h^n, c_h^n) \right\}_{n=0}^N$ is used.
The results in Tables \ref{tab:est_cond_1} to \ref{tab:2nd_order_Es} show the EOCs that expected from our analysis. Note that since the integrand $a^{\xi}$ incorporates the broken $H^1$-norms of $\left\{ \rho_h^n \right\}_{n=0}^N$, the full error estimator is sensitive to the gradient of the exact solution.

\begin{remark} [Round-off errors]
    Considering the precision limitations of the \texttt{float64} data type in NumPy, we observe that the $L^2$ norm of the jump of $\{c_h^n\}_{n=0}^N$ cannot decrease below a certain threshold. Specifically, the squared jump of $\{c_h^n\}_{n=0}^N$ is constrained to a minimum of $10^{-16}$, which aligns with the machine epsilon of \texttt{float64} in NumPy, approximately $2.22 \times 10^{-16}$. As a result, the estimator $\tilde{E}_1[\bar{c}^{\tau}, \bar{\rho}^{\tau}]$ and subsequently $E_{R^{\xi}}$ experience a decline of the EOC due to round-off errors; see the row $i=8$ of Table \ref{tab:2nd_order_Es_tilde}.
\end{remark}

Thus the full error estimator is expected to scale with order $h^k+\tau$ as long as the exact solution is sufficiently smooth and round-off errors do not dominate.

\begin{table} [!]
    \centering
    \begin{tabular}{@{}lccc@{}}
        \toprule
        $i$ & $\int_0^T a^{\xi}(t)dt$ & $E_0[\rho_h^0]$ & EOC \\
        \midrule
        $4$ & 1.63E+08 & 4.33E+02 & 2.068 \\
        $5$ & 1.61E+08 & 1.03E+02 & 2.051 \\
        $6$ & 9.10E+07 & 2.49E+01 & 2.022 \\
        $7$ & 3.13E+07 & 6.14E+00 & 2.010 \\
        $8$ & 9.42E+06 & 1.52E+00 & 2.005 \\
        $9$ & 3.43E+06 & 3.80E-01 & - \\
        \bottomrule
    \end{tabular}      
    \caption{Polynomial degree $k = 1$.
        Tests were conducted for $\int_0^T a^{\xi}(t)dt$ and $E_0[\rho_h^0]$ using different values of $ i $. The end time is $T=0.0045$. Note that since the integrand $ a^{\xi} $ incorporates the broken $ H^1 $-norms of $ \left\{ \rho_h^n \right\}_{n=0}^N $, it is sensitive to the gradient of the exact solution. The EOCs for $E_0[\rho_h^0]$ approach the theoretical order.
    }
    \label{tab:est_cond_1}
\end{table}

\begin{table} [!]
    \centering
    \begin{tabular}{@{}lccc@{}}
        \toprule
        $i$ & $\int_0^T a^{\xi}(t)dt$ & $E_0[\rho_h^0]$ & EOC \\
        \midrule
        $4$ & 8.00E+08 & 1.77E+02 & 2.936 \\
        $5$ & 1.81E+08 & 2.31E+01 & 2.898 \\
        $6$ & 1.35E+07 & 3.10E+00 & 2.915 \\
        $7$ & 2.24E+06 & 4.12E-01 & 2.944 \\
        $8$ & 1.50E+06 & 5.35E-02 & 2.967 \\
        $9$ & 1.43E+06 & 6.84E-03 & - \\
        \bottomrule
    \end{tabular}    
    
    \caption{Polynomial degree $k = 2$.
        Tests were conducted for $\int_0^T a^{\xi}(t)dt$ and $E_0[\rho_h^0]$ using different values of $ i $. The end time is $T=0.0045$. Note that since the integrand $ a^{\xi} $ incorporates the broken $ H^1 $-norms of $ \left\{ \rho_h^n \right\}_{n=0}^N $, it is sensitive to the gradient of the exact solution. The EOCs for $E_0[\rho_h^0]$ approach the theoretical order.
    }
    \label{tab:est_cond_2}
\end{table}

\begin{table} [!]
    \centering
    \begin{tabular}{@{}lcccccc@{}}
        \toprule
        $i$ & $\left\| \operatorname{E}_{0} \right\|_{L^\infty(0,T)}$ & EOC & $\left\| \operatorname{E}_{1} \right\|_{L^2(0,T)}$ & EOC & $\left\| \operatorname{E}_{R^{\xi}} \right\|_{L^2(0,T)}$ & EOC \\
        \midrule
        $4$ & 7.14E+01 & 1.015 & 8.08E+02 & 0.015 & 1.78E+06 & 0.615 \\
        $5$ & 3.54E+01 & 1.416 & 8.00E+02 & 0.416 & 1.16E+06 & 1.450 \\
        $6$ & 1.33E+01 & 1.793 & 5.99E+02 & 0.793 & 4.26E+05 & 1.918 \\
        $7$ & 3.82E+00 & 1.952 & 3.46E+02 & 0.952 & 1.13E+05 & 1.627 \\
        $8$ & 9.88E-01 & 1.995 & 1.79E+02 & 0.995 & 3.65E+04 & 1.243 \\
        $9$ & 3.78E-01 & -    & 8.98E+01     & -    & 1.54E+04 & -    \\
        \bottomrule
    \end{tabular}

    \caption{Polynomial degree $k = 1$. The experimental results for the norms of $E_{0}[\bar{\rho}^{\tau}(t, \cdot)]$, $E_{1}[\bar{\rho}^{\tau}(t, \cdot)]$, $E_{R^{\xi}}[\bar{\rho}^{\tau}(t, \cdot)]$ are presented. }
    \label{tab:1st_order_Es_part1}
\end{table}

% \begin{table} 
%     \centering
%     \begin{tabular}{@{}lcccc@{}}
%     \toprule
%     $N$ & $\bar{A}$ & EOC & $\bar{E}$ & EOC \\
%     \midrule
%     $16$ & 1.67E+01 & 1.814 & 1.89E+02 & 0.814 \\
%     $32$ & 4.74E+00 & 1.944 & 1.07E+02 & 0.944 \\
%     $64$ & 1.23E+00 & 1.983 & 5.57E+01 & 0.983 \\
%     $128$ & 3.12E-01 & 1.994 & 2.82E+01 & 0.994 \\
%     $256$ & 7.82E-02 & - & 1.42E+01 & - \\
%     \bottomrule
%     \end{tabular} 
    % \caption{Polynomial degree $k = 1$. Experiment results on $\bar{A}$ and $\bar{E}$ in log scale. The sequence of solutions is computed with a mesh size $h = 2^{-N}$, time step $\tau = 2^{2-N}$ and $T = 0.0005$.}
    % \label{tab:1st_order_Es_part2}
% \end{table}

\begin{table} [!]
    \centering
    \begin{tabular}{@{}lcccccc@{}}
        \toprule
        $i$ & $\left\| \operatorname{E}_{0} \right\|_{L^\infty(0,T)}$ & EOC & $\left\| \operatorname{E}_{1} \right\|_{L^2(0,T)}$ & EOC & $\left\| \operatorname{E}_{R^{\xi}} \right\|_{L^2(0,T)}$ & EOC \\
        \midrule
        $4$ & 1.58E+02 & 2.077 & 1.79E+03 & 1.077 & 1.22E+07 & 2.777 \\
        $5$ & 3.75E+01 & 2.959 & 8.48E+02 & 1.959 & 1.78E+06 & 4.429 \\
        $6$ & 4.82E+00 & 3.021 & 2.18E+02 & 2.021 & 8.25E+04 & 3.577 \\
        $7$ & 5.94E-01 & 3.024 & 5.37E+01 & 2.024 & 6.92E+03 & 2.577 \\
        $8$ & 7.30E-02 & 3.013 & 1.32E+01 & 2.013 & 1.16E+03 & 1.669 \\
        $9$ & 3.78E-01 & -    & 3.27E+00 & -    & 3.65E+02 & -    \\
        \bottomrule
    \end{tabular}

    \caption{Polynomial degree $k=2$. The experimental results for the norms of $E_{0}[\bar{\rho}^{\tau}(t, \cdot)]$, $E_{1}[\bar{\rho}^{\tau}(t, \cdot)]$, $E_{R^{\xi}}[\bar{\rho}^{\tau}(t, \cdot)]$ are presented. }
    \label{tab:2nd_order_Es} 
\end{table}

\begin{table}[!ht]
    \centering
    \begin{minipage}[t]{0.45\linewidth}
        \centering
        \begin{tabular}{@{}lcc@{}}
            \toprule
            $i$ & $\left\| E_{0}[\partial_t\bar{\rho}^{\tau}] \right\|_{L^2(0,T)}$ & EOC \\
            \midrule
            $4$ & 1.95E+04 & 0.275 \\
            $5$ & 1.61E+04 & 1.012 \\
            $6$ & 7.98E+03 & 1.621 \\
            $7$ & 2.59E+03 & 1.899 \\
            $8$ & 6.95E+02 & 1.980 \\
            $9$ & 1.76E+02 & - \\
            \bottomrule
        \end{tabular}
        \caption{Polynomial degree $k=1$. The experimental results for the norm of $E_{0}[\partial_t\bar{\rho}^{\tau}(t, \cdot)]$ are presented.} \label{tab:time_derivative_Es}
    \end{minipage}%
    \hfill
    \begin{minipage}[t]{0.45\linewidth}
        \centering
        \begin{tabular}{@{}lcc@{}}
            \toprule
            $i$ & $\left\| E_{-1}[\partial_t\bar{\rho}^{\tau}] \right\|_{L^2(0,T)}$ & EOC \\
            \midrule
            $4$ & 7.57E+03 & 2.483 \\
            $5$ & 1.36E+03 & 3.782 \\
            $6$ & 9.85E+01 & 4.010 \\
            $7$ & 6.12E+00 & 4.022 \\
            $8$ & 3.76E-01 & 4.011 \\
            $9$ & 2.33E-02 & - \\
            \bottomrule
        \end{tabular}
        \caption{Polynomial degree $k=2$. The experimental results for the norm of $E_{-1}[\partial_t\bar{\rho}^{\tau}(t, \cdot)]$ are presented.} \label{tab:time_derivative_Es_2}
    \end{minipage}
\end{table}

\begin{table}[!ht]
    \centering
    \begin{minipage}[t]{0.45\linewidth}
        \centering
        \begin{tabular}{@{}lcc@{}}
            \toprule
            $i$ & $\left\| \operatorname{\tilde{E}}_{1} \right\|_{L^2(0,T)}$ & EOC \\
            \midrule
            $4$ & 1.24E+00 & 0.779 \\
            $5$ & 7.22E-01 & 0.915 \\
            $6$ & 3.83E-01 & 0.989 \\
            $7$ & 1.93E-01 & 1.008 \\
            $8$ & 9.59E-02 & 1.008 \\
            $9$ & 4.77E-02 & - \\
            \bottomrule
        \end{tabular}
        \caption{Polynomial degree $k=1$. The experimental results for the norm of $\tilde{E}_{1}[\bar{c}^{\tau}(t, \cdot)]$ are presented.}
        % \label{tab:1st_order_Es_tilde}
    \end{minipage}%
    \hfill
    \begin{minipage}[t]{0.45\linewidth}
        \centering
        \begin{tabular}{@{}lcc@{}}
            \toprule
            $i$ & $\left\| \operatorname{\tilde{E}}_{1} \right\|_{L^2(0,T)}$ & EOC \\
            \midrule
            $4$ & 6.42E-01 & 1.878 \\
            $5$ & 1.75E-01 & 2.065 \\
            $6$ & 4.18E-02 & 2.056 \\
            $7$ & 1.00E-02 & 1.992 \\
            $8$ & 2.52E-03 & 1.065 \\
            $9$ & 1.21E-03 & - \\
            \bottomrule     
        \end{tabular} 
        \caption{Polynomial degree $k=2$. The experimental results for the norm of $\tilde{E}_{1}[\bar{c}^{\tau}(t, \cdot)]$ are presented. This indicates a reduction of the EOC at $i=8$.}
        \label{tab:2nd_order_Es_tilde}
    \end{minipage}
\end{table}

% \section{Conclusion}
% In this paper, we provide a stability estimate framework for the Keller-Segel system. Using this framework,
% we have derived a posteriori error estimates for a fully discrete scheme, based on the semi-discrete discontinuous Galerkin scheme of the Keller-Segel system.
% The error estimator is shown to be optimal in the sense that it scales with the expected rate of convergence of the dG scheme. The numerical experiments confirm the theoretical results and show that the estimator behaves as expected. 
% % The motivation of this work is to provide a reliable and efficient error estimator for the Keller-Segel system, which can be used to guide adaptive mesh refinement strategies.
% } {\tt Kiwoong: need to mention the future work for the adaptive mesh algorithm based on our error estimator as a future work?}

\newpage
% \section*{Statements and Declarations}
\paragraph{Funding} J.G. is grateful for financial support by the German Science Foundation 
(DFG) via grant TRR 154 (\emph{Mathematical modelling, simulation and 
optimization using the example of gas networks}), project C05. The work 
of J.G. is also supported by the Graduate School CE within Computational 
Engineering at Technische Universität Darmstadt.
K. Kwon was supported by the National Research Foundation of Korea (NRF) grant funded by the Korea government (MSIT) (No. 2020R1A4A1018190, 2021R1C1C1011867).

\section*{Declarations}
\paragraph{Conflict of interest} The authors declare no competing interests.

% \section*{Acknowledgments}

% This work began during the second author's visit to TU Darmstadt in the summer of 2022. The second author would like to express his gratitude to Prof. Min-Gi Lee and Prof. Philsu Kim from Kyungpook National University, Korea, for making the visit possible. 

\newpage
\bibliographystyle{abbrvurl}

\bibliography{./references}

\end{document}